\documentclass[11pt,a4paper]{amsart}
\usepackage{amsmath,amssymb, amsbsy}
\usepackage{color,psfrag}
\usepackage[dvips]{graphicx}
\usepackage{enumerate}
\newcommand{\R}{{\mathbb R}}

\newcommand{\eps}{\varepsilon}

\renewcommand{\ge }{\geqslant}
\renewcommand{\geq }{\geqslant}
\renewcommand{\le }{\leqslant}
\renewcommand{\leq }{\leqslant}
\def\neweq#1{\begin{equation}\label{#1}}
\def\endeq{\end{equation}}
\def\eq#1{(\ref{#1})}
\newtheorem{theorem}{Theorem}[section]
\newtheorem{proposition}[theorem]{Proposition}
\newtheorem{lemma}[theorem]{Lemma}
\newtheorem{corollary}[theorem]{Corollary}

\newtheorem{remark}[theorem]{Remark}

\theoremstyle{definition}

\renewcommand{\arraystretch}{1.5}
\begin{document}

\title[variation of frequencies in a rectangular plate]{On the variation of longitudinal and torsional frequencies\\
in a partially hinged rectangular plate}

\author[Elvise BERCHIO]{Elvise BERCHIO}
\address{\hbox{\parbox{5.7in}{\medskip\noindent{Dipartimento di Scienze Matematiche, \\
Politecnico di Torino,\\
        Corso Duca degli Abruzzi 24, 10129 Torino, Italy. \\[3pt]
        \em{E-mail address: }{\tt elvise.berchio@polito.it}}}}}
\author[Davide BUOSO]{Davide BUOSO}
\address{\hbox{\parbox{5.7in}{\medskip\noindent{Dipartimento di Scienze Matematiche, \\
Politecnico di Torino,\\
        Corso Duca degli Abruzzi 24, 10129 Torino, Italy. \\[3pt]
        \em{E-mail address: }{\tt davide.buoso@polito.it}}}}}
\author[Filippo GAZZOLA]{Filippo GAZZOLA}
\address{\hbox{\parbox{5.7in}{\medskip\noindent{Dipartimento di Matematica,\\
Politecnico di Milano,\\
   Piazza Leonardo da Vinci 32, 20133 Milano, Italy. \\[3pt]
        \em{E-mail addresses: }{\tt
          filippo.gazzola@polimi.it}}}}}

\date{\today}

\keywords{Shape variation; eigenvalues; plates; torsional instability; suspension bridges}

\subjclass[2010]{35J40; 35P15; 74K20}

\begin{abstract}
We consider a partially hinged rectangular plate and its normal modes. There are two families of modes, longitudinal and torsional.
We study the variation of the corresponding eigenvalues under domain deformations. We investigate the possibility of finding a shape functional able to
quantify the torsional instability of the plate, namely how prone is the plate to transform longitudinal oscillations into torsional ones. This functional
should obey several rules coming from both theoretical and practical evidences. We show that a simple functional obeying all the required rules does not
exist and that the functionals available in literature are not reliable.
\end{abstract}

\maketitle

\section{Introduction}

We consider a thin rectangular elastic plate $\Omega$ which is hinged on two opposite edges and free on the remaining two edges. Thanks to scaling, we may
restrict our attention to the plate $\Omega=(0,\pi)\times(-\ell,\ell)$ where the width $2\ell$ is assumed to be much smaller than the length $\pi$,
that is, $2\ell\ll\pi$. In this plate we study the following eigenvalue problem
\begin{equation}\label{bridge}
\begin{cases}
\Delta^2 u=\lambda u & \qquad \text{in } \Omega\,, \\
u(0,y)=u_{xx}(0,y)=u(\pi,y)=u_{xx}(\pi,y)=0 & \qquad \text{for } y\in (-\ell,\ell)\,, \\
u_{yy}(x,\pm\ell)+\sigma
u_{xx}(x,\pm\ell)=u_{yyy}(x,\pm\ell)+(2-\sigma)u_{xxy}(x,\pm\ell)=0
& \qquad \text{for } x\in (0,\pi)\, ,
\end{cases}
\end{equation}
where $\sigma$ denotes the Poisson ratio of the material forming the plate. For most elastic materials one has $0<\sigma<0.5$;
since we aim to model the deck of a bridge, which is a mixture of concrete and steel, we will take $\sigma=0.2$. The boundary conditions on the short
edges tell that the deck is hinged and this models the connection of the deck with the ground; these conditions are named after Navier since their first
appearance in \cite{navier}. The boundary conditions on the large edges model the fact that the deck is free to move vertically and they
may also be written in the equivalent form (cf.\ \cite{chasman})
\neweq{equivalent}
(1-\sigma)\frac{\partial^2 u}{\partial\nu^2}+\sigma\Delta u=\frac{\partial\Delta u}{\partial\nu}
+(1-\sigma)\mathrm{div}_{\partial\Omega}(\nu\cdot D^2u)_{\partial\Omega}=0\qquad\mbox{on }(0,\pi)\times\{-\ell,\ell\}.
\endeq

In two recent papers \cite{bfg,fergaz}, the eigenvalue problem \eq{bridge} was analyzed in order to study the stability of suspension bridges. It turns out
that these bridges, as well as any elastic plate modeled by \eq{bridge}, display two kinds of oscillations: the longitudinal and the torsional ones.
In Section \ref{longtors} we emphasize their classification and the corresponding shapes of oscillations (the related eigenfunctions, that is, the normal modes).
We also compute numerically some of the eigenvalues.\par
There is a growing interest of engineers on the shape optimization for the design of bridges and, in particular, on the sensitivity analysis of certain
eigenvalue problems, see \cite[Chapter 6]{jhnm}. As pointed out by Banerjee \cite{banerjee}, {\em the free vibration analysis is a fundamental
pre-requisite before carrying out a flutter analysis}. Whence, since the eigenvalues of \eq{bridge} are the squared frequencies of the normal
modes of $\Omega$, in order to improve the stability of the plate one has to analyze the behavior of the eigenvalues with respect to
perturbations of the domain $\Omega$; since we have in mind the application to bridges, we cannot vary the two short edges of the plate and, therefore,
we will consider domain variations which do not preserve its area. Classical references for the behavior of
the eigenvalues of an elliptic operator under domain perturbation are \cite{bucur,delfour,henrot,henrotpierre,henry,sok}.
In Theorem \ref{implicitfunction} we establish the explicit formula for the derivatives of the eigenvalues when the perturbation of $\Omega$ merely
consists in varying its width $2\ell$; in fact, for later use, we prove a result allowing to compute the derivative of any smooth function
depending on couples of eigenvalues. Concerning shape deformations leading to geometries different from a rectangle, a major difficulty for problem
\eq{bridge} is that it is a fourth order equation and this fact usually yields very complicated formulas for the derivatives of the eigenvalues,
in particular the boundary condition \eq{equivalent} is very delicate: in Theorem \ref{duesetted} we overcome this difficulty by taking advantage of
recent results obtained in his doctoral dissertation by the second author \cite{buosotesi}.\par
Then we apply these results to a stability problem. The most dangerous oscillations for the deck of a bridge, leading to fractures and collapses,
are the torsional ones and the target of engineers is to find possible ways to prevent their appearance; we refer to
Section \ref{thresholds} for detailed explanations. We investigate the possibility of defining a domain functional able to quantify how much a plate is prone to
transform longitudinal oscillations into torsional ones. Such a functional should depend on the particular couple of (longitudinal,torsional) modes
considered, see \cite[Section 2.5]{banerjee}.\par
Together with some colleagues, the first and the third author made several attempts to
mathematically describe the torsional instability for different bridge models, see \cite{bbfg,bfg,elvisefilippo,bgz} and also \cite{bookgaz} for a
survey of the available results. The conclusion of these works is that the instability depends not only on the couple of torsional and
longitudinal frequencies involved but also on the amount of energy present within the structure. This result is reached by studying some second order 2-DOF
Hamiltonian systems that naturally arise while approximating the PDE modeling the dynamics of the deck.
For this reason, in Section \ref{criticalenergy} we study some prototype Hamiltonian systems and we compute their energy thresholds for stability.
As far as we are aware, there is no similar systematic study in literature. It turns out that the energy threshold is very sensitive to the (nonlinear) coupling terms and
it appears very hard to derive a general rule. However, the studied Hamiltonian systems also share a common feature: their energy threshold
only depends on the {\em ratio} of the two frequencies considered.\par
In Section \ref{nostro}, we combine the above facts with some ``axioms'' from the engineering literature, that is, some fundamental properties that are to be expected from bridges and
that are obtained either from experimental tests or from actual bridges.
Then, by exploiting the explicit form of the shape
derivatives previously found in Section \ref{dompert}, we investigate whether some shape functionals fit the obtained requirements. These functionals do not simply depend on the eigenvalues but also on the shape of the domain:
in these situations, the analysis usually requires a suitable combination of variational methods from shape optimization and numerical methods, see
e.g.\ \cite{antunes,henrotoudet,ksw,oudet} for some examples. We show that some of these functionals partially fulfill the basic rules
of a shape functional aiming to compute the energy threshold for the torsional stability. But none of these functionals works for all the
couples of (longitudinal,torsional) eigenvalues. Our conclusion is that there exists no general and simple shape functional obeying all these rules and able to
measure the stability of bridges.\par
In Section \ref{flutter} we quickly survey a part of the engineering literature where one can find some attempts to derive thresholds for the stability
of the deck. By applying again Theorems \ref{implicitfunction} and \ref{duesetted},
we show that also the most popular formula used to compute the energy threshold does not satisfy the fundamental properties observed in actual bridges.
Sections \ref{proof1}-\ref{lemmass} are devoted to the proofs of the results stated in Sections \ref{longtors} and \ref{dompert} while in Section \ref{conc} we summarize the results and we draw our conclusions.

\section{Longitudinal and torsional eigenfunctions and eigenvalues}\label{longtors}

The natural functional space where to study problem \eq{bridge} is
$$
H^2_*(\Omega)=\big\{u\in H^2(\Omega): u=0\mathrm{\ on\ }\{0,\pi\}\times(-\ell,\ell)\big\},
$$
that we endow with the scalar product and corresponding norm
\begin{equation}\label{neumannform1}
(u,v)=\int_\Omega (1-\sigma) D^2u:D^2v+\sigma\Delta u\Delta v\, dA\qquad\forall u,v\in H^2_*(\Omega)\ ,\qquad\|u\|^2=(u,u)\,,
\end{equation}
where $D^2u:D^2v=u_{xx}v_{xx}+2u_{xy}v_{xy}+u_{yy}v_{yy}$ and $dA$ denotes the area element.
Then problem (\ref{bridge}) may also be formulated in the following weak sense
\begin{equation*}
\int_\Omega (1-\sigma) D^2u:D^2v+\sigma\Delta u\Delta v dA=\lambda\int_{\Omega}uv dA\qquad\forall v\in H^2_*(\Omega).
\end{equation*}

First we slightly improve \cite[Theorem 7.6]{fergaz} (see also
\cite[Proposition 3.1]{bfg}) by showing that the eigenfunctions of \eq{bridge} may have one of the forms listed below.

\begin{theorem}\label{eigenvalue}
The set of eigenvalues of \eqref{bridge} may be ordered in an increasing sequence of strictly positive numbers diverging to $+\infty$ and any eigenfunction belongs to
$C^\infty(\overline\Omega)$; the set of eigenfunctions of \eqref{bridge} is a complete system in $H^2_*(\Omega)$. Moreover:\par\noindent
$(i)$ for any $m\ge1$, there exists a unique eigenvalue $\lambda=\mu_{m,1}\in((1-\sigma^2)m^4,m^4)$ with corresponding eigenfunction
$$\left[\big[\mu_{m,1}^{1/2}-(1-\sigma)m^2\big]\, \tfrac{\cosh\Big(y\sqrt{m^2+\mu_{m,1}^{1/2}}\Big)}{\cosh\Big(\ell\sqrt{m^2+\mu_{m,1}^{1/2}}\Big)}+
\big[\mu_{m,1}^{1/2}+(1-\sigma)m^2\big]\, \tfrac{\cosh\Big(y\sqrt{m^2-\mu_{m,1}^{1/2}}\Big)}{\cosh\Big(\ell\sqrt{m^2-\mu_{m,1}^{1/2}}\Big)}\right]\sin(mx)\, ;$$
$(ii)$ for any $m\ge1$ and any $k\ge2$ there exists a unique eigenvalue $\lambda=\mu_{m,k}>m^4$ satisfying
$\left(m^2+\frac{\pi^2}{\ell^2}\left(k-\frac{3}{2}\right)^2\right)^2<\mu_{m,k}<\left(m^2+\frac{\pi^2}{\ell^2}\left(k-1\right)^2\right)^2$
and with corresponding eigenfunction
$$
\left[\big[\mu_{m,k}^{1/2}-(1-\sigma)m^2\big]\, \tfrac{\cosh\Big(y\sqrt{\mu_{m,k}^{1/2}+m^2}\Big)}{\cosh\Big(\ell\sqrt{\mu_{m,k}^{1/2}+m^2}\Big)}
+\big[\mu_{m,k}^{1/2}+(1-\sigma)m^2\big]\, \tfrac{\cos\Big(y\sqrt{\mu_{m,k}^{1/2}-m^2}\Big)}{\cos\Big(\ell\sqrt{\mu_{m,k}^{1/2}-m^2}\Big)}\right]\sin(mx)\, ;
$$
$(iii)$ for any $n\ge1$ and any $j\ge2$ there exists a unique eigenvalue $\lambda=\nu_{n,j}>n^4$ with corresponding eigenfunctions
$$
\left[\big[\nu_{n,j}^{1/2}-(1-\sigma)n^2\big]\, \tfrac{\sinh\Big(y\sqrt{\nu_{n,j}^{1/2}+n^2}\Big)}{\sinh\Big(\ell\sqrt{\nu_{n,j}^{1/2}+n^2}\Big)}
+\big[\nu_{n,j}^{1/2}+(1-\sigma)n^2\big]\, \tfrac{\sin\Big(y\sqrt{\nu_{n,j}^{1/2}-n^2}\Big)}{\sin\Big(\ell\sqrt{\nu_{n,j}^{1/2}-n^2}\Big)}\right]\sin(nx)\, ;
$$
$(iv)$ for any $n\ge1$ satisfying $\ell n\sqrt 2\, \coth(\ell n\sqrt2 )>\left(\frac{2-\sigma}{\sigma}\right)^2$ there exists a unique
eigenvalue $\lambda=\nu_{n,1}\in(\mu_{n,1},n^4)$ with corresponding eigenfunction
$$\left[\big[\nu_{n,1}^{1/2}-(1-\sigma)n^2\big]\, \tfrac{\sinh\Big(y\sqrt{n^2+\nu_{n,1}^{1/2}}\Big)}{\sinh\Big(\ell\sqrt{n^2+\nu_{n,1}^{1/2}}\Big)}
+\big[\nu_{n,1}^{1/2}+(1-\sigma)n^2\big]\, \tfrac{\sinh\Big(y\sqrt{n^2-\nu_{n,1}^{1/2}}\Big)}{\sinh\Big(\ell\sqrt{n^2-\nu_{n,1}^{1/2}}\Big)}\right]
\sin(nx)\, .$$
Finally, if the unique positive solution $s>0$ of the equation
\neweq{iff}
\tanh(\sqrt{2}s\ell)=\left(\frac{\sigma}{2-\sigma}\right)^2\, \sqrt{2}s\ell
\endeq
is not an integer, then the only eigenvalues and eigenfunctions are the ones given in $(i)-(iv)$.
\end{theorem}

A sketch of the proof of the improved bounds in Theorem \ref{eigenvalue} is given in Section \ref{proof1}. Condition \eq{iff} has probability 0 to occur in
general plates; if it occurs, there is an additional eigenvalue and eigenfunction, see \cite{fergaz}. From \cite{bfg,fergaz} we recall that the eigenvalues are solutions of explicit equations.

\begin{proposition}\label{defunzioni}
Let{\small
$$
\Phi^m(\lambda,\ell):=\sqrt{m^2\!-\!\lambda^{1/2}}\big(\lambda^{1/2}\!+\!(1\!-\!\sigma)m^2\big)^2\tanh(\ell\sqrt{m^2\!-\!\lambda^{1/2}})\!
-\!\sqrt{m^2\!+\!\lambda^{1/2}}\big(\lambda^{1/2}\!-\!(1\!-\!\sigma)m^2\big)^2\tanh(\ell\sqrt{m^2\!+\!\lambda^{1/2}}),
$$
$$
\Upsilon^m(\lambda,\ell):=\sqrt{\lambda^{1/2}\!-\!m^2}\big(\lambda^{1/2}\!+\!(1\!-\!\sigma)m^2\big)^2\tan(\ell\sqrt{\lambda^{1/2}\!-\!m^2})\!+
\!\sqrt{\lambda^{1/2}\!+\!m^2}\big(\lambda^{1/2}\!-\!(1\!-\!\sigma)m^2\big)^2\tanh(\ell\sqrt{\lambda^{1/2}\!+\!m^2}),
$$
$$
\Psi^n(\lambda,\ell):=\sqrt{\lambda^{1/2}\!-\!n^2}\big(\lambda^{1/2}\!+\!(1\!-\!\sigma)n^2\big)^2\tanh(\ell\sqrt{\lambda^{1/2}\!+\!n^2})\!-
\!\sqrt{\lambda^{1/2}\!+\!n^2}\big(\lambda^{1/2}\!-\!(1\!-\!\sigma)n^2\big)^2\tan(\ell\sqrt{\lambda^{1/2}\!-\!n^2}),
$$
$$
\Gamma^n(\lambda,\ell):=\sqrt{n^2\!-\!\lambda^{1/2}}\big(\lambda^{1/2}\!+\!(1\!-\!\sigma)n^2\big)^2\tanh(\ell\sqrt{\lambda^{1/2}\!+\!n^2})\!-
\!\sqrt{\lambda^{1/2}\!+\!n^2}\big(\lambda^{1/2}\!-\!(1\!-\!\sigma)n^2\big)^2\tanh(\ell\sqrt{\lambda^{1/2}\!-\!n^2}).
$$}
Then:\par\noindent
$(i)$ the eigenvalue $\lambda=\mu_{m,1}$ is the unique value $\lambda\in((1-\sigma^2)m^4,m^4)$ such that $\Phi^m(\lambda,\ell)=0$;\par\noindent
$(ii)$ the eigenvalues $\lambda=\mu_{m,k}$ ($k\ge2$) are the solutions $\lambda>m^4$ of the equation $\Upsilon^m(\lambda,\ell)=0$;\par\noindent
$(iii)$ the eigenvalues $\lambda=\nu_{n,j}$ ($j\ge2$) are the solutions $\lambda>n^4$ of the equation $\Psi^n(\lambda,\ell)=0$;\par\noindent
$(iv)$ the eigenvalue $\lambda=\nu_{n,1}$ is the unique value $\lambda\in((1-\sigma^2)n^4,n^4)$ such that $\Gamma^n(\lambda,\ell)=0$.
\end{proposition}

The eigenfunctions in $(i)-(ii)$ are even with respect to $y$ whereas the eigenfunctions in $(iii)-(iv)$ are odd. We call {\bf longitudinal
eigenfunctions} the eigenfunctions of the kind $(i)-(ii)$ and {\bf torsional eigenfunctions} the eigenfunctions of the kind $(iii)-(iv)$.
Since $\ell$ is small, the former are quite similar to $c_m\sin(mx)$ whereas the latter are similar to $c_ny\sin(nx)$.\par
In the sequel, we consider realistic values of $\sigma$ and $\ell$, as in some actual bridges; we take
\neweq{parameters}
\sigma=0.2\, ,\qquad\ell=\frac{\pi}{150},
\endeq
but very similar results are obtained for values of $\sigma$ and $\ell$ close to \eq{parameters}. This choice of $\ell$ models the case where the
main span of the bridge is $1$ kilometer long and the width $2\ell$ is about $13$ meters.
These values are taken from the original Tacoma Narrows Bridge, see \cite{ammann,bfg}. We denote by
\neweq{denote}
\bar\mu_{m,k}\ \mbox{and}\ \bar\nu_{n,j}\ \mbox{the eigenvalues of \eq{bridge} given in Theorem \ref{eigenvalue} and Proposition
\ref{defunzioni} when \eqref{parameters} holds.}
\endeq
Then, from Theorem \ref{eigenvalue}, we infer that
$$0.96 m^4<\bar\mu_{m,1}<m^4\, ,\qquad(m^2+75^2(2k-3)^2)^2<\bar\mu_{m,k}<(m^2+150^2(k-1)^2)^2\quad\forall k\ge2$$
for all integer $m$.
Furthermore, a direct inspection yields that
\neweq{num1}
\bar\nu_{n,1}\quad \text{does not exist for } 1\leq n\leq 2734\,.
\endeq

In Table \ref{tableigen} we collect some numerical values of $\bar\mu_{m,k}$ and $\bar\nu_{n,j}$ as defined in \eq{denote}.

\begin{table}[ht]
\begin{center}

{\footnotesize
\begin{tabular}{|c|c|c|c|c|c|c|c|c|c|c|c|c|c|}
\hline
$\bar\mu_{1,1}$\!&\!$\bar\mu_{2,1}$\!&\!$\bar\mu_{3,1}$\!&\!$\bar\mu_{4,1}$\!&\!$\bar\mu_{5,1}$\!&\!$\bar\mu_{6,1}$\!&\!$\bar\mu_{7,1}$\!&\!$\bar\mu_{8,1}$\!&\!$\bar\mu_{9,1}$\!&\!$\bar\mu_{10,1}$\!&\!$\bar\mu_{11,1}$\!&\!$\bar\mu_{12,1}$\!&\!$\bar\mu_{13,1}$\!&\!$\bar\mu_{14,1}$
\\
\hline
$0.96$\!&\!$15.36$\!&\!$77.77$\!&\!$245.8$\!&\!$600.14$\!&\!$1244.6$\!&\!$2306.05$\!&\!$3934.57$\!&\!$6303.42$\!&\!$9609.09$\!&\!$14071.4$\!&\!$19933.4$\!&\!$27461.6$\!&\!$36946$\\
\hline
\end{tabular}

\begin{tabular}{|c|c|c|c|c|}
\hline
$\bar\nu_{1,2}$\!&\!$\bar\nu_{2,2}$\!&\!$\bar\nu_{3,2}$\!&\!$\bar\nu_{4,2}$\!&\!$\bar\nu_{5,2}$\\
\hline
$10943.63$\!&\!$43785.82$\!&\!$98560.47$\!&\!$175324.1$\!&\!$274155.8$ \\
\hline
\end{tabular}

\begin{tabular}{|c|c|c|c|c|c|c|c|c|c|c|c|c|c|c|}
\hline
$\bar\mu_{1,2}$\!&\!$\bar\mu_{2,2}$\!&\!$\bar\mu_{3,2}$\!&\!$\bar\mu_{4,2}$\!&\!$\bar\mu_{5,2}$\!&\!$\bar\mu_{6,2}$\!&\!$\bar\mu_{7,2}$\!&\!$\bar\mu_{8,2}$\!&\!$\bar\mu_{9,2}$\!&\! $\bar\mu_{10,2}$\!&\!$\bar\mu_{11,2}$\!&\!$\bar\mu_{12,2}$\!&\!$\bar\mu_{13,2}$\!&\!$\bar\mu_{14,2}$\!&\!$\times$\\
\hline
$1.626$\!&\!$1.628$\!&\!$1.63$\!&\!$1.634$\!&\!$1.638$\!&\!$1.643$\!&\!$1.649$\!&\!$1.657$\!&\!$1.665$\!&\!$1.674$\!&\!$1.684$\!&\!$1.695$\!&\!$1.707$\!&\!$1.72$\!&\!$10^8$\\
\hline
\end{tabular}

\begin{tabular}{|c|c|c|c|c|}
\hline
$\bar\nu_{1,3}$\!&\!$\bar\nu_{2,3}$\!&\!$\bar\nu_{3,3}$\!&\!$\bar\nu_{4,3}$\!&\!$\bar\nu_{5,3}$\\
\hline
$1.2356\!\cdot\!10^9$\!&\!$1.2359\!\cdot\!10^9$\!&\!$1.2365\!\cdot\!10^9$\!&\!$1.2372\!\cdot\!10^9$\!&\!$1.2382\!\cdot\!10^9$\\
\hline
\end{tabular}
}
\caption{Numerical values of some eigenvalues of problem \eq{bridge} when \eq{parameters} holds.}\label{tableigen}
\end{center}
\end{table}

These results are fairly precise and reliable. The ``exact'' value of these eigenvalues will be important in the following sections. Here we just
point out that
\begin{equation}
\label{comparison}
\bar\mu_{1,1}<\dots<\bar\mu_{10,1}<\bar\nu_{1,2}<\bar\mu_{11,1}<\dots<\bar\mu_{14,1}<\bar\nu_{2,2},
\end{equation}
\begin{equation}\label{comparison2}
\bar\nu_{n,2}<\bar\mu_{m,2}<\bar\nu_{n,3}\qquad\mbox{for all }m=1,...,14\mbox{ and }n=1,...,5\, .
\end{equation}

In fact, we considered all the $k=1,2,3,4$, and $j=2,3,4,5$, for $\bar\mu_{m,k}$ and $\bar\nu_{n,j}$ with $m\le14$ and $n\le5$. Let us briefly
summarize what we observed numerically.\par
$\bullet$ The map $m\mapsto\bar\mu_{m,1}$ is strictly increasing and $0.96<\bar\mu_{m,1}<36946.004$ for $m=1,...,14$.\par
$\bullet$ The map $m\mapsto\bar\mu_{m,2}$ is strictly increasing and $1.62\cdot10^8<\bar\mu_{m,2}<1.721\cdot10^8$ for $m=1,...,14$.\par
$\bullet$ The map $m\mapsto\bar\mu_{m,3}$ is strictly increasing and $4.74\cdot10^9<\bar\mu_{m,3}<4.786\cdot10^9$ for $m=1,...,14$.\par
$\bullet$ The map $m\mapsto\bar\mu_{m,4}$ is strictly increasing and $2.895\cdot10^{10}<\bar\mu_{m,4}<2.904\cdot10^{10}$ for $m=1,...,14$.\par
$\bullet$ The map $n\mapsto\bar\nu_{n,2}$ is strictly increasing and $10943.6<\bar\nu_{n,2}<274155.9$ for $n=1,...,5$.\par
$\bullet$ The map $n\mapsto\bar\nu_{n,3}$ is strictly increasing and $1.235\cdot10^9<\bar\nu_{n,3}<1.239\cdot10^9$ for $n=1,...,5$.\par
$\bullet$ The map $n\mapsto\bar\nu_{n,4}$ is strictly increasing and $1.297\cdot10^{10}<\bar\nu_{n,4}<1.299\cdot10^{10}$ for $n=1,...,5$.\par
$\bullet$ The map $n\mapsto\bar\nu_{n,5}$ is strictly increasing and $5.648\cdot10^{10}<\bar\nu_{n,5}<5.65\cdot10^{10}$ for $n=1,...,5$.\par\bigskip

In terms of the frequencies (the square roots of the eigenvalues) the above observations show that
\begin{equation}\label{large freq}
\mbox{\bf the smallest frequencies of the normal modes are those listed in Table \ref{tableigen}.}
\end{equation}
These facts explain why we mainly restricted our attention to the eigenvalues in Table \ref{tableigen} (i.e. $k=1,2$ and $j=2,3$). Moreover, the eigenvalues $\bar\mu_{m,2}$ are much bigger than the eigenvalues $\bar\mu_{m,1}$, and this translates in larger frequencies. This means that a bigger amount of energy is needed in order to trigger the normal modes associated with $\bar\mu_{m,2}$, so that it is quite unlikely to observe them. The same remark holds also for $\bar\nu_{n,2},\bar\nu_{n,3}$. \par
Note that the restrictions $m\le14$ and $n\le5$ are not just motivated by the lack of space in this paper but also by the behavior in actual bridges; at the collapsed
Tacoma Narrows Bridge the longitudinal oscillations appeared with at most ten nodes and the torsional oscillation appeared with one node, see \cite{ammann} and Section \ref{conc} below.\par
Finally, by \eq{num1} we know that the torsional eigenvalues $\bar\nu_{n,1}$ do not exist for $n\leq 2734$, while for $n\geq 2735$ the frequencies are very large.

\section{Domain perturbations and variation of the frequencies}\label{dompert}

The aim of this section is to study the variation of the longitudinal and torsional frequencies when the
rectangular plate $\Omega$ changes width or, more generally, shape.\par
We start by considering the effect of the variation of the width: we assume \eq{parameters} and we use the notations \eq{denote}.
By the Implicit Function Theorem, the relation $\Phi^m(\lambda,\ell)=0$ implicitly defines, in a neighborhood $U$ of $\ell=\frac{\pi}{150}$, a
smooth function $\mu_{m,1}=\mu_{m,1}(\ell)$ such that
$$\mu_{m,1}\left(\frac{\pi}{150}\right)=\bar\mu_{m,1}\, ,\qquad\Phi^m\big(\mu_{m,1}(\ell),\ell\big)=0\quad\forall\ell\in U\, .$$
Similarly, we define the smooth functions $\mu_{m,k}=\mu_{m,k}(\ell)$, $\nu_{n,1}=\nu_{n,1}(\ell)$, $\nu_{n,j}=\nu_{n,j}(\ell)$.
Then, we further exploit the Implicit Function Theorem to derive the following

\begin{theorem}\label{implicitfunction}
Let $\sigma=0.2$ and let $\mu_{m,k}(\ell)$ and $\nu_{n,j}(\ell)$ be the functions defined above. Furthermore, let $\Phi^m,\Upsilon^m,\Psi^n,\Gamma^n$ be
as defined in Proposition \ref{defunzioni} and $\Phi^m_\ell$ and $\Phi^m_\lambda$ denote the partial derivatives of $\Phi^m$ and similarly for $\Upsilon^m,\Psi^n$ and $\Gamma^n$. If $(\mu,\nu)\in \R^2 \mapsto f(\mu,\nu)\in \R$ is a differentiable map, then for all positive integers $m,k,n,j$ the functions
$$\ell\mapsto f(\mu_{m,k}(\ell),\nu_{n,j}(\ell))$$
are differentiable and their derivatives for $\ell=\pi/150$ are given by (here $j,k\ge2$)

$$\frac{d}{d \ell}\left(f(\mu_{m,1}(\ell),\nu_{n,1}(\ell))\right)(\tfrac{\pi}{150})=
-f_{\mu}(\bar\mu_{m,1},\bar\nu_{n,1})\frac{\Phi^m_\ell(\bar\mu_{m,1},\frac{\pi}{150})}{\Phi^m_\lambda(\bar\mu_{m,1},\frac{\pi}{150})}-
f_{\nu}(\bar\mu_{m,1},\bar\nu_{n,1})\frac{\Gamma^n_\ell(\bar\nu_{n,1},\frac{\pi}{150})}{\Gamma^n_\lambda(\bar\nu_{n,1},\frac{\pi}{150})},$$

$$\frac{d}{d \ell}\left(f(\mu_{m,k}(\ell),\nu_{n,1}(\ell))\right)(\tfrac{\pi}{150})=
-f_{\mu}(\bar\mu_{m,k},\bar\nu_{n,1})\frac{\Upsilon^m_\ell(\bar\mu_{m,k},\frac{\pi}{150})}{\Upsilon^m_\lambda(\bar\mu_{m,k},\frac{\pi}{150})}-
f_{\nu}(\bar\mu_{m,k},\bar\nu_{n,1})\frac{\Gamma^n_\ell(\bar\nu_{n,1},\frac{\pi}{150})}{\Gamma^n_\lambda(\bar\nu_{n,1},\frac{\pi}{150})},$$

$$\frac{d}{d \ell}\left(f(\mu_{m,1}(\ell),\nu_{n,j}(\ell))\right)(\tfrac{\pi}{150})=
-f_{\mu}(\bar\mu_{m,1},\bar\nu_{n,j})\frac{\Phi^m_\ell(\bar\mu_{m,1},\frac{\pi}{150})}{\Phi^m_\lambda(\bar\mu_{m,1},\frac{\pi}{150})}-
f_{\nu}(\bar\mu_{m,1},\bar\nu_{n,j})\frac{\Psi^n_\ell(\bar\nu_{n,j},\frac{\pi}{150})}{\Psi^n_\lambda(\bar\nu_{n,j},\frac{\pi}{150})},$$

$$\frac{d}{d \ell}\left(f(\mu_{m,k}(\ell),\nu_{n,j}(\ell))\right)(\tfrac{\pi}{150})=
-f_{\mu}(\bar\mu_{m,k},\bar\nu_{n,j})\frac{\Upsilon^m_\ell(\bar\mu_{m,k},\frac{\pi}{150})}{\Upsilon^m_\lambda(\bar\mu_{m,k},\frac{\pi}{150})}-
f_{\nu}(\bar\mu_{m,k},\bar\nu_{n,j})\frac{\Psi^n_\ell(\bar\nu_{n,j},\frac{\pi}{150})}{\Psi^n_\lambda(\bar\nu_{n,j},\frac{\pi}{150})},$$
where $f_{\mu}$ and $f_{\nu}$ denote the partial derivatives of $f$, $\bar\mu_{m,1}=\mu_{m,1}(\frac{\pi}{150})$, $\bar\mu_{m,k}=\mu_{m,k}(\frac{\pi}{150})$, $\bar\nu_{n,1}=\nu_{n,1}(\frac{\pi}{150})$, $\bar\nu_{n,j}=\nu_{n,j}(\frac{\pi}{150})$, see \eqref{denote} .
\end{theorem}

In view of \eq{num1}, only the last two derivatives will be useful for our purposes. Furthermore, Theorem \ref{implicitfunction} has the following immediate
consequence.

\begin{corollary}\label{implicitfunctioncorollary}
Under the same assumptions of Theorem \ref{implicitfunction}, for all positive integers $m,k,n,j$ the functions
$$\ell\mapsto \mu_{m,k}(\ell)\quad \text{and}\quad \ell\mapsto \nu_{n,j}(\ell)$$
are differentiable and their derivatives for $\ell=\pi/150$ are given by (here $j,k\ge2$)
$$\frac{d}{d \ell}\left(\mu_{m,1}(\ell)\right)(\tfrac{\pi}{150})=
-\frac{\Phi^m_\ell(\bar\mu_{m,1},\frac{\pi}{150})}{\Phi^m_\lambda(\bar\mu_{m,1},\frac{\pi}{150})}\,,\quad\frac{d}{d \ell}\left(\nu_{n,1}(\ell)\right)(\tfrac{\pi}{150})=-\frac{\Gamma^n_\ell(\bar\nu_{n,1},\frac{\pi}{150})}{\Gamma^n_\lambda(\bar\nu_{n,1},\frac{\pi}{150})}\,,$$

$$\frac{d}{d \ell}\left(\mu_{m,k}(\ell)\right)(\tfrac{\pi}{150})=-\frac{\Upsilon^m_\ell(\bar\mu_{m,k},\frac{\pi}{150})}{\Upsilon^m_\lambda(\bar\mu_{m,k},\frac{\pi}{150})}\,,\quad
\frac{d}{d \ell}\left(\nu_{n,j}(\ell)\right)(\tfrac{\pi}{150})=-\frac{\Psi^n_\ell(\bar\nu_{n,j},\frac{\pi}{150})}{\Psi^n_\lambda(\bar\nu_{n,j},\frac{\pi}{150})},$$
where $\bar\mu_{m,1}=\mu_{m,1}(\frac{\pi}{150})$, $\bar\mu_{m,k}=\mu_{m,k}(\frac{\pi}{150})$, $\bar\nu_{n,1}=\nu_{n,1}(\frac{\pi}{150})$, $\bar\nu_{n,j}=\nu_{n,j}(\frac{\pi}{150})$.
\end{corollary}

Now we turn to the effect of the variation of the shape of the plate.
We consider problem (\ref{bridge}) in a family of open sets  parameterized by suitable diffeomorphisms $\xi$
defined on $\Omega$, by maintaining fixed the short edges $\Gamma_1=\{0,\pi\}\times(-\ell,\ell)$. Namely, we set
$$
{\mathcal{A}}_{\Omega }=\biggl\{\xi\in C^2_b(\Omega\, ; {\mathbb{R}}^2 ):\ \inf_{\substack{p_1,p_2\in \Omega \\ p_1\ne p_2}}
\frac{|\xi(p_1)-\xi(p_2)|}{|p_1-p_2|}>0,\ \xi(p)=p\ \forall p\in\Gamma_1 \biggr\},
$$
where $C^2_b(\Omega\, ; {\mathbb{R}}^2 )$ denotes the space of functions from $\Omega $ to ${\mathbb{R}}^2$ of class $C^2$, with bounded derivatives up to
order $2$. We recall that since $\partial\Omega\in C^{0,1}$ then $C^2_b(\Omega\, ; {\mathbb{R}}^2 )\subset C^0(\overline{\Omega}\, ; {\mathbb{R}}^2 )$.
Note that if $\xi \in {\mathcal{A}}_{\Omega }$ then $\xi $ is injective, Lipschitz continuous and $\inf_{\Omega }|{\rm det }\nabla \xi |>0$;
we denote by $\nabla$ both the gradient (vector) of a scalar function and the Jacobian (matrix) of a vector function. Moreover,
$\xi(\Omega )$ is a bounded open set and the inverse map $\xi^{-1}$ belongs to  ${\mathcal{A}}_{\xi(\Omega )}$. Then we define
$$
H^2_*(\xi(\Omega))=\{u\in H^2(\xi(\Omega)): u\circ\xi\in H^2_*(\Omega)\}.
$$

In view of \eq{equivalent} we may write problem (\ref{bridge}) as follows
\neweq{bridge2}
\Delta^2 u\!=\!\lambda u \mbox{ in }\xi(\Omega),\ u\!=\!u_{xx}\!=\!0\mbox{ on }\Gamma_1,\
(1\!-\!\sigma)\frac{\partial^2 u}{\partial\nu^2}+\sigma\Delta u\!=\!\frac{\partial\Delta u}{\partial\nu}
+(1\!-\!\sigma)\mathrm{div}_{\partial\xi(\Omega)}(\nu\cdot D^2u)_{\partial\xi(\Omega)}\!=\!0\mbox{ on }\Gamma_2,
\endeq
where $\Gamma_2=\xi(\partial\Omega\setminus\Gamma_1)$.
We consider problem (\ref{bridge2}) and study the dependence of the eigenvalues $\lambda[\xi (\Omega )]$ on $\xi \in {\mathcal{A}}_{\Omega }$.
We endow the space $C^2_b(\Omega\, ; {\mathbb{R}}^2 )$ with its usual norm $\|f\|_{C^2_b(\Omega\, ;{\mathbb{R}}^2)}=\sup_{|\alpha|\le2,\ p\in\Omega}|D^{\alpha }f(p)|$. Note that ${\mathcal{A}}_{\Omega }$ is an open set in  $C_b^2(\Omega\, ;{\mathbb{R}}^2 )$, see \cite[Lemma~3.11]{lala2004}.
Thus, we may study differentiability and analyticity properties of the maps $\xi \mapsto \lambda[\xi (\Omega )]$ defined for
$\xi\in{\mathcal{A}}_{\Omega }$.
We intend to deform only the free edges of the deck, therefore we investigate deformations of the form
\begin{equation}\label{psi}
\psi(x,y)=\left(x,\tau(x)+y(\delta(x)+1)\right),
\end{equation}
where $\tau,\delta\in C^2[0,\pi]$ are such that $\tau(0)=\tau(\pi)=\delta(0)=\delta(\pi)=0$. Let ${\mathcal{A}}_\lambda=\{\xi\in {\mathcal{A}}_{\Omega };\, \lambda[\xi(\Omega)]\mbox{ is simple}\}$.
Then we have the following result, whose proof is given in Section \ref{dimostrazione}.

\begin{theorem}\label{duesetted}
The set ${\mathcal{A}}_\lambda$ is open in $\mathcal A_{\Omega}$, and the map $\xi\mapsto\lambda[\xi(\Omega)]$ from $\mathcal A_\lambda$
to $\mathbb{R}$ is real analytic.
Moreover, if $v$ is an eigenfunction of $\lambda[\Omega]$ normalized with respect to the scalar product (\ref{neumannform1}), and $\psi$ is as in \eqref{psi}, then we have the following formula for the Fr\'{e}chet differential
\neweq{derivd3}
d|_{\xi={\rm Id}}\lambda[\xi(\Omega)][\psi-{\rm Id}]=2\ell\lambda[\Omega]\int_0^\pi\left((1-\sigma)|D^2v|^2+
\sigma(\Delta v)^2-\lambda[\Omega]v^2\right)|_{y=\ell}\delta(x) dx.
\endeq
\end{theorem}

We observe that the derivative in formula (\ref{derivd3}) does not depend on $\tau$. The reason is that $\tau$ acts in ``opposite'' ways on the two long edges. Therefore, from now on we will take $\tau(x)\equiv0$. In this case,
$$
\psi(x,y)-{\rm Id}(x,y)=\left(0,y\ \delta(x)\right), \ \forall (x,y)\in\overline{\Omega}.
$$
In particular, $\pm\ell\delta$ represent the variations of the edges $y=\pm\ell$. For later use, we set
\neweq{phi_def}
\phi(x)=\ell\delta(x)\quad \forall x\in(0,\pi).
\endeq

Since $\phi(0)=\phi(\pi)=0$, we may expand $\phi$ in its Fourier series,
$$
\phi(x)=\sum_{h=1}^\infty a_h\sin(hx),
$$
and analyze the effects term by term.
We have the following corollary, where we consider deformations of the type
$\phi(x)=h\sin(hx)$. We use these deformations (with coefficient $h$) in order to have the same area increment with respect to the original rectangle,
a fact which enables us to compare the results.

\begin{corollary}\label{esplicito}
Under the same hypotheses of Theorem \ref{duesetted}, we have
\neweq{esplicita}
d|_{\xi={\rm Id}}\lambda[\xi(\Omega)][(0,\tfrac{h}{\ell}y\sin(hx))]=2h\lambda[\Omega]\int_0^\pi\left((1-\sigma)|D^2v|^2+
\sigma(\Delta v)^2-\lambda[\Omega]v^2\right)|_{y=\ell}\sin(hx) dx.
\endeq
\end{corollary}

Note that if $h$ is even, then the derivative \eqref{esplicita} vanishes,
and therefore produces no differences for the shape deformation problem. This can be seen using Lemmas \ref{calcoli} and \ref{torsionali}, coupled with the following elementary equalities
$$
\int_0^\pi\sin^2(mx)\sin(hx)dx=\frac 1 2\int_0^\pi\sin(hx)dx-\frac 1 2 \int_0^\pi\cos(2mx)\sin(hx)dx=\frac{(-1)^{h+1}+1}{2}\frac{4m^2}{h(4m^2-h^2)}
$$
and
$$
\int_0^\pi\cos^2(mx)\sin(hx)dx=\frac 1 2\int_0^\pi\sin(hx)dx+\frac 1 2 \int_0^\pi\cos(2mx)\sin(hx)dx=\frac{(-1)^{h+1}+1}{2}\frac{4m^2-2h^2}{h(4m^2-h^2)}
$$
which are both zero if $h$ is even. We used here the fact that
$$
\int_0^\pi\cos(ax)\sin(bx)dx=b\,\frac{(-1)^{a+b}-1}{a^2-b^2}\qquad\forall a,b\in\mathbb{N},\  a\neq b.
$$
Whence, we will concentrate on odd sines only.

\begin{remark}\label{notation_phi}
When $\Omega$ simply changes width, definition \eq{phi_def} has to be thought with $\phi(x)=1$. We observe that the perturbations $\phi(x)=1$
and $\phi(x)=h\sin (hx)$ should not be compared, since their behavior on the shape is different. Indeed, contrary to the latter, the former does not
preserve the hinged edges. Also, the respective formulas for the derivatives of the eigenvalues are obtained in different ways, see Theorems \ref{implicitfunction} and \ref{duesetted}.
\end{remark}

We conclude this section with some numerical computations of the derivatives of the eigenvalues considered in \eq{comparison}-\eq{comparison2}. For the sake of readability, we denote by $D_h(\cdot)$ the derivative in \eqref{esplicita}, for any eigenvalue (or
combination of eigenvalues). With abuse of notation, we also denote by $D_\ell(\cdot)$ the derivative with respect to the width, as in Theorem \ref{implicitfunction}.

\begin{table}[ht]
\begin{center}
{\tiny
\begin{tabular}{|c|c|c|c|c|c|c|c|c|c|c|c|c|c|c|}
\hline
$ $\!\!&\!\!$\mu_{1,1}$\!\!&\!\!$\mu_{2,1}$\!\!&\!\!$\mu_{3,1}$\!\!&\!\!$\mu_{4,1}$\!\!&\!\!$
\mu_{5,1}$\!\!&\!\!$\mu_{6,1}$\!\!&\!\!$\mu_{7,1}$\!\!&\!\!$\mu_{8,1}$\!\!&\!\!$\mu_{9,1}$\!\!&\!\!$\mu_{10,1}$\!\!&\!\!$\mu_{11,1}$\!\!&\!\!$\mu_{12,1}$\!\!&\!\!$\mu_{13,1}$\!\!&\!\!$\mu_{14,1}$\\
\hline
$D_{\ell}$\!\!&\!\!$89\!\!\cdot\!\!10^{-5}$\!\!&\!\!$57\!\!\cdot\!\!10^{-3}$\!\!&\!\!$65\!\!\cdot\!\!10^{-2}$\!\!&\!\!$3.6$\!\!&\!\!$13.7$\!\!&\!\!$40.8$\!\!&\!\!$102.2$\!\!&\!\!$225.7$\!\!&\!\!$453.2$\!\!&\!\!$843.6$\!\!&\!\!$1476.8$\!\!&\!\!$2457.2$\!\!&\!\!$3917$\!\!&\!\!$6019.6$\\
\hline
$D_1$\!\!&\!\!$19\!\!\cdot\!\!10^{-5}$\!\!&\!\!$31\!\!\cdot\!\!10^{-3}$\!\!&\!\!$39\!\!\cdot\!\!10^{-2}$\!\!&\!\!$2.23$\!\!&\!\!$8.58$\!\!&\!\!$25.6$\!\!&\!\!$64.4$\!\!&\!\!$143$\!\!&\!\!$287$\!\!&\!\!$534$\!\!&\!\!$936$\!\!&\!\!$16\!\!\cdot\!\!10^{2}$\!\!&\!\!$25\!\!\cdot\!\!10^{2}$\!\!&\!\!$38\!\!\cdot\!\!10^{2}$\\
\hline
$D_3$\!\!&\!\!$26\!\!\cdot\!\!10^{-4}$\!\!&\!\!$-57\!\!\cdot\!\!10^{-3}$\!\!&\!\!$13\!\!\cdot\!\!10^{-2}$\!\!&\!\!$1.55$\!\!&\!\!$7.02$\!\!&\!\!$22.5$\!\!&\!\!$58.7$\!\!&\!\!$133$\!\!&\!\!$272$\!\!&\!\!$512$\!\!&\!\!$904$\!\!&\!\!$15\!\!\cdot\!\!10^{2}$\!\!&\!\!$24\!\!\cdot\!\!10^{2}$\!\!&\!\!$37\!\!\cdot\!\!10^{2}$\\
\hline
$D_5$\!\!&\!\!$19\!\!\cdot\!\!10^{-4}$\!\!&\!\!$24\!\!\cdot\!\!10^{-2}$\!\!&\!\!$-1.47$\!\!&\!\!$-0.66$\!\!&\!\!$2.89$\!\!&\!\!$15.0$\!\!&\!\!$45.9$\!\!&\!\!$112$\!\!&\!\!$240$\!\!&\!\!$464$\!\!&\!\!$836$\!\!&\!\!$14\!\!\cdot\!\!10^{2}$\!\!&\!\!$23\!\!\cdot\!\!10^{2}$\!\!&\!\!$36\!\!\cdot\!\!10^{2}$\\
\hline
\end{tabular}

\begin{tabular}{|c|c|c|c|c|c|c|c|c|c|c|c|c|c|c|c|}
\hline
$ $\!\!&\!\!$\mu_{1,2}$\!\!&\!\!$\mu_{2,2}$\!\!&\!\!$\mu_{3,2}$\!\!&\!\!$\mu_{4,2}$\!\!&\!\!$
\mu_{5,2}$\!\!&\!\!$\mu_{6,2}$\!\!&\!\!$\mu_{7,2}$\!\!&\!\!$\mu_{8,2}$\!\!&\!\!$\mu_{9,2}$\!\!&\!\!$\mu_{10,2}$\!\!&\!\!$\mu_{11,2}$\!\!&\!\!$\mu_{12,2}$\!\!&\!\!$\mu_{13,2}$\!\!&\!\!$\mu_{14,2}$\!\!&\!\! $\times\downarrow$\\
\hline
$D_\ell$\!\!&\!\!$-3.106$\!\!&\!\!$-3.107$\!\!&\!\!$-3.109$\!\!&\!\!$-3.113$\!\!&\!\!$-3.117$\!\!&\!\!$-3.122$\!\!&\!\!$-3.128$\!\!&\!\!$-3.135$\!\!&\!\!$-3.142$\!\!&\!\!$-3.151$\!\!&\!\!$-3.16$\!\!&\!\!$-3.17$\!\!&\!\!$-3.18$\!\!&\!\!$-3.19$\!\!&\!\!$10^{10}$\\
\hline
$D_1$\!\!&\!\!$-2.636$\!\!&\!\!$-2.110$\!\!&\!\!$-2.036$\!\!&\!\!$-2.013$\!\!&\!\!$-2.004$\!\!&\!\!$-2.001$\!\!&\!\!$-2.002$\!\!&\!\!$-2.004$\!\!&\!\!$-2.007$\!\!&\!\!$-2.011$\!\!&\!\!$-2.016$\!\!&\!\!$-2.022$\!\!&\!\!$-2.029$\!\!&\!\!$-2.037$\!\!&\!\!$10^{10}$\\
\hline
$D_3$\!\!&\!\!$1.583$\!\!&\!\!$-4.524$\!\!&\!\!$-2.641$\!\!&\!\!$-2.307$\!\!&\!\!$-2.182$\!\!&\!\!$-2.121$\!\!&\!\!$-2.088$\!\!&\!\!$-2.069$\!\!&\!\!$-2.059$\!\!&\!\!$-2.053$\!\!&\!\!$-2.051$\!\!&\!\!$-2.052$\!\!&\!\!$-2.055$\!\!&\!\!$-2.059$\!\!&\!\!$10^{10}$\\
\hline
$D_5$\!\!&\!\!$0.377$\!\!&\!\!$3.522$\!\!&\!\!$-6.488$\!\!&\!\!$-3.257$\!\!&\!\!$-2.650$\!\!&\!\!$-2.409$\!\!&\!\!$-2.286$\!\!&\!\!$-2.215$\!\!&\!\!$-2.171$\!\!&\!\!$-2.143$\!\!&\!\!$-2.125$\!\!&\!\!$-2.114$\!\!&\!\!$-2.107$\!\!&\!\!$-2.104$\!\!&\!\!$10^{10}$\\
\hline
\end{tabular}

\begin{tabular}{|c|c|c|c|c|c|c|c|c|c|c|}
\hline
$ $\!\!&\!\!$\nu_{1,2}$\!\!&\!\!$\nu_{2,2}$\!\!&\!\!$\nu_{3,2}$\!\!&\!\!$\nu_{4,2}$\!\!&\!\!$
\nu_{5,2}$\!\!&\!\!$\nu_{1,3}$\!\!&\!\!$\nu_{2,3}$\!\!&\!\!$\nu_{3,3}$\!\!&\!\!$\nu_{4,3}$\!\!&\!\!$\nu_{5,3}$\\
\hline
$D_\ell$\!\!&\!\!$-10^6$\!\!&\!\!$-42\!\!\cdot\!\!10^5$\!\!&\!\!$-94\!\!\cdot\!\!10^5$\!\!&\!\!$-17\!\!\cdot\!\!10^6$\!\!&\!\!$-26\!\!\cdot\!\!10^6
$\!\!&\!\!$-24\!\!\cdot\!\!10^{10}$\!\!&\!\!$-24\!\!\cdot\!\!10^{10}$\!\!&\!\!$-24\!\!\cdot\!\!10^{10}$\!\!&\!\!$-24\!\!\cdot\!\!10^{10}$\!\!&\!\!$-24\!\!\cdot\!\!10^{10}$\\
\hline
$D_1$\!\!&\!\!$-11\!\!\cdot\!\!10^5$\!\!&\!\!$-30\!\!\cdot\!\!10^5$\!\!&\!\!$-63\!\!\cdot\!\!10^5$\!\!&\!\!$-11\!\!\cdot\!\!10^6$\!\!&\!\!$-17\!\!\cdot\!\!10^6$\!\!&\!\!$
-20\!\!\cdot\!\!10^{10}$\!\!&\!\!$-16\!\!\cdot\!\!10^{10}$\!\!&\!\!$-15\!\!\cdot\!\!10^{10}$\!\!&\!\!$-15\!\!\cdot\!\!10^{10}$\!\!&\!\!$-15\!\!\cdot\!\!10^{10}$\\
\hline
$D_3$\!\!&\!\!$17\!\!\cdot\!\!10^5$\!\!&\!\!$-95\!\!\cdot\!\!10^5$\!\!&\!\!$-10\!\!\cdot\!\!10^6$\!\!&\!\!$-14\!\!\cdot\!\!10^6$\!\!&\!\!$-20\!\!\cdot\!\!10^6$\!\!&\!\!$12\!\!\cdot\!\!10^{10}$\!\!&\!\!$-34\!\!\cdot\!\!10^{10}$\!\!&\!\!$-20\!\!\cdot\!\!10^{10}$\!\!&\!\!$-17\!\!\cdot\!\!10^{10}$\!\!&\!\!$-17\!\!\cdot\!\!10^{10}$\\
\hline
$D_5$\!\!&\!\!$92\!\!\cdot\!\!10^4$\!\!&\!\!$12\!\!\cdot\!\!10^6$\!\!&\!\!$-33\!\!\cdot\!\!10^6$\!\!&\!\!$-24\!\!\cdot\!\!10^6$\!\!&\!\!$-28\!\!\cdot\!\!10^6$\!\!&\!\!$
29\!\!\cdot\!\!10^{9}$\!\!&\!\!$27\!\!\cdot\!\!10^{10}$\!\!&\!\!$-49\!\!\cdot\!\!10^{10}$\!\!&\!\!$-25\!\!\cdot\!\!10^{10}$\!\!&\!\!$-20\!\!\cdot\!\!10^{10}$\\
\hline
\end{tabular}

\par\medskip
\caption{Numerical values of the derivatives of some eigenvalues when \eq{parameters} holds.}\label{tableder}}
\end{center}
\end{table}

\begin{table}[ht]
\begin{center}
{\tiny
\begin{tabular}{|c|c|c|c|c|c|c|c|c|c|c|c|c|c|c|}
\hline
$ $\!\!&\!\!$\mu_{1,1}$\!\!&\!\!$\mu_{2,1}$\!\!&\!\!$\mu_{3,1}$\!\!&\!\!$\mu_{4,1}$\!\!&\!\!$
\mu_{5,1}$\!\!&\!\!$\mu_{6,1}$\!\!&\!\!$\mu_{7,1}$\!\!&\!\!$\mu_{8,1}$\!\!&\!\!$\mu_{9,1}$\!\!&\!\!$\mu_{10,1}$\!\!&\!\!$\mu_{11,1}$\!\!&\!\!$\mu_{12,1}$\!\!&\!\!$\mu_{13,1}$\!\!&\!\!$\mu_{14,1}$\\
\hline
$\frac{\nu_{2,2}}{\mu_{m,1}}$\!\!&\!\!$45609.8$\!\!&\!\!$ 2850.53$\!\!&\!\!$ 563.04$\!\!&\!\!$ 178.14$\!\!&\!\!$ 72.96$\!\!&\!\!$35.18$\!\!&\!\!$ 18.99$\!\!&\!\!$11.13$\!\!&\!\!$ 6.95$\!\!&\!\!$ 4.56$\!\!&\!\!$ 3.11$\!\!&\!\!$ 2.2$\!\!&\!\!$ 1.6$\!\!&\!\!$ 1.18$\\
\hline

\end{tabular}

\begin{tabular}{|c|c|c|c|c|c|c|c|c|c|c|c|c|c|c|c|}
\hline
$ $\!\!&\!\!$\frac{\nu_{2,2}}{\mu_{1,1}}$\!\!&\!\!$\frac{\nu_{2,2}}{\mu_{2,1}}$\!\!&\!\!$\frac{\nu_{2,2}}{\mu_{3,1}}$\!\!&\!\!$\frac{\nu_{2,2}}{\mu_{4,1}}$\!\!&\!\!$\frac{\nu_{2,2}}{\mu_{5,1}}$\!\!&\!\!$\frac{\nu_{2,2}}{\mu_{6,1}}$\!\!&\!\!$\frac{\nu_{2,2}}{\mu_{7,1}}$\!\!&\!\!$\frac{\nu_{2,2}}{\mu_{8,1}}$\!\!&\!\!$\frac{\nu_{2,2}}{\mu_{9,1}}$\!\!&\!\!$\frac{\nu_{2,2}}{\mu_{10,1}}$\!\!&\!\!$\frac{\nu_{2,2}}{\mu_{11,1}}$\!\!&\!\!$\frac{\nu_{2,2}}{\mu_{12,1}}$\!\!&\!\!$\frac{\nu_{2,2}}{\mu_{13,1}}$\!\!&\!\!$\frac{\nu_{2,2}}{\mu_{14,1}}$\!\!&\!\! $\times\downarrow$\\
\hline
$D_\ell$\!\!&\!\!$-4.3\cdot 10^4$\!\!&\!\!$ -2721.2$\!\!&\!\!$ -537.5$\!\!&\!\!$ -170.1$\!\!&\!\!$ -69.6$\!\!&\!\!$ -33.6$\!\!&\!\!$-18.1$\!\!&\!\!$ -10.6$\!\!&\!\!$ -6.6$\!\!&\!\!$ -4.3$\!\!&\!\!$ -3$\!\!&\!\!$ -2.01$\!\!&\!\!$ -1.5$\!\!&\!\!$ -1.1$\!\!&\!\! $10^2$\\
\hline
$D_1$\!\!&\!\!$-31\!\!\cdot\!\!10^3$\!\!&\!\!$-20\!\!\cdot\!\!10^2$\!\!&\!\!$-388$\!\!&\!\!$-123$\!\!&\!\!$-50.3$\!\!&\!\!$-24.2$\!\!&\!\!$-13.1$\!\!&\!\!$-7.67$\!\!&\!\!$-4.79$\!\!&\!\!$-3.14$\!\!&\!\!$-2.15$\!\!&\!\!$-1.51$\!\!&\!\!$-1.10$\!\!&\!\!$-0.82$\!\!&\!\!$10^2$\\
\hline
$D_3$\!\!&\!\!$-99\!\!\cdot\!\!10^3$\!\!&\!\!$-62\!\!\cdot\!\!10^2$\!\!&\!\!$-12\!\!\cdot\!\!10^2$\!\!&\!\!$-387$\!\!&\!\!$-158$\!\!&\!\!$-76.4$\!\!&\!\!$-41.2$\!\!&\!\!$-24.2$\!\!&\!\!$-15.1$\!\!&\!\!$-9.89$\!\!&\!\!$-6.76$\!\!&\!\!$-4.77$\!\!&\!\!$-3.46$\!\!&\!\!$-2.57$\!\!&\!\!$10^2$\\
\hline

\end{tabular}

\par\medskip
\caption{Numerical values of the ratios $\frac{\nu_{2,2}}{\mu_{m,1}}$ and their derivatives when \eq{parameters} holds.}\label{tableder2}}
\end{center}
\end{table}

The numerical values in Tables \ref{tableder} and \ref{tableder2} will be quite useful in the sequel.  We see that
\begin{equation}
\label{osservazione}
\mbox{the derivatives with respect to $\phi=1$ (}D_\ell\mbox{) and to $\phi=\sin x$ (}D_1\mbox{) have the same sign.}
\end{equation}
We also emphasize the following facts.\par
$\bullet$ The absolute value of the derivatives with respect to $\phi=1$ is increasing with respect to $m$, $n$, $k$, and $j$; this is not
fully visible for $\nu_{n,3}$ in Table \ref{tableder} but more precise numerical values confirm this behavior.\par
$\bullet$ When $\phi=1$, the derivatives of $\mu_{m,1}$ are all positive and they are strictly increasing with respect to $m$, whereas
the derivatives of $\mu_{m,2}$ are all negative and they are strictly decreasing with respect to $m$; the latter behavior is also visible
for $\mu_{m,3}$ and $\mu_{m,4}$.\par
$\bullet$ When $\phi=1$, the derivatives of $\nu_{n,j}$ ($j=2,3$) are all negative and they are strictly decreasing with respect to $n$;
this behavior is also visible for $\nu_{n,4}$ and $\nu_{n,5}$.\par
$\bullet$ For $\phi(x)=3\sin3x,\, 5\sin5x\, ,$ the derivatives do not display a clear behavior neither with respect to the absolute value nor with respect to the
sign; this is due to the combination of the sine appearing in $\phi$ with the sine appearing in the eigenfunctions, see Theorem \ref{eigenvalue}. The same happens also for $\phi(x)=7\sin7x,\, 9\sin9x,\, 11\sin11x$.
It is clear that a rule exists but its determination falls beyond the scopes of the present paper.\par
$\bullet$ We also compared the values of the ratios $\gamma(m)=\frac{\nu_{2,2}}{\mu_{m,1}}$ for $m=1,\dots,14$ with their derivatives
$D_k\gamma(m)=D_k(\tfrac{\nu_{2,2}}{\mu_{m,1}})$ for $k=\ell,1,3$. In all three cases, we deduce from Table \ref{tableder2} that these ratios and their derivatives
almost perfectly obey to the following linear law
\begin{equation}\label{odeigen}
\tfrac{D_k\gamma(m)}{c_{0,k}+c_{1,k}\,\gamma(m)}=-1,
\end{equation}
for $k=\ell,1,3$ and $m=1,\dots,14$, with
$$
c_{0,\ell}= 0.14\,,\, c_{1,\ell}=95.53\,,\quad
c_{0,1}=1.897\cdot10^{-3} \,,\, c_{1,1}= 1.443\,,\quad
c_{0,3}= 19\cdot10^{-4}\,,\, c_{1,3}=4.546 \,.
$$
These facts will be exploited in Section \ref{nostro}.

\section{Thresholds for the torsional stability}\label{thresholds}

In recent years the attention of civil engineers has shifted towards the sensitivity analysis and optimal design aiming to improve the performances of
suspension bridges, see \cite{hhs,jhnm}. As explicitly mentioned in the preface of the monograph by Jurado, Hern\'andez, Nieto, and Mosquera \cite{jhnm},
the trend is nowadays to avoid expensive tests in wind tunnels and to test numerically the performances of different designs; hopefully, these tests
should be preceded by a suitable mathematical modeling and, possibly, by analytic arguments, see \cite{lacarbonara}.\par
According to the Federal Report \cite{ammann}, the main reason for the Tacoma Narrows Bridge collapse \cite{tacoma} was the sudden transition from
longitudinal to torsional oscillations. Several other bridges collapsed for the same reason, see e.g.\ \cite[Chapter 1]{bookgaz} or the introduction
in \cite{bfg}. Hence, the most dangerous oscillations, leading to fractures and collapses, are the torsional ones and a common target of engineers is to
find possible ways to prevent their appearance in the deck. Our purpose in this second part of the paper is to discuss the possibility
of finding a domain functional able to quantify how much a plate is prone to transform longitudinal oscillations into torsional ones.
To this end, we need to review some rules that a deck is known to obey, starting from the aerodynamic mechanism generating oscillations.\par
When the wind hits the deck of a bridge the direction of the flow is modified and goes around the body. Behind the deck, or a corner of it, the flow
creates vortices which are, in general, asymmetric. This asymmetry generates a forcing lift which launches the vertical oscillations of the deck. Up to
some minor details, this explanation is shared by the whole community and has been studied with great precision in wind tunnel tests, see e.g.\
\cite{larsen,scanlan,scott}.

Inspired by previous results by Bleich \cite{bleichsolo,bleich}, Rocard \cite[p.163]{rocard} shows that for common bridges there exists a threshold for the
velocity of the wind (that we denote by $V_c$) above which the bridge undergoes to {\em flutter}, namely a form of instability which is visible in many
objects and appears as an uncontrolled vibration, see the videos \cite{youtube}. On the other hand, the vortices induced by the wind increase the
internal energy of the structure and generate wide longitudinal oscillations which look periodic in time and are maintained in amplitude by a somehow
perfect equilibrium between the input of energy from the wind and internal dissipation: at this point, one may assume that the deck is isolated.
Then one wonders how the longitudinal oscillations
may transform into torsional ones. Recent results in \cite{bbfg,bfg,elvisefilippo,bgz} show that the transition is due to some internal
resonance, a phenomenon typical of (conservative) Hamiltonian systems, see \cite{SVM,verh}. This resonance appears in nonlinear systems and depends on
the amount of energy inside the system: when an energy threshold (that we denote by $E_c$) is reached, there is a sudden transfer of energy between
the different components of the Hamiltonian system. For bridges this translates into the possible appearance of wide torsional oscillations from an
apparently pure longitudinal oscillation.\par
In fact, the thresholds for flutter and for internal resonance are linked by the well-known formula for the kinetic energy
\neweq{EcVc}
E_c=k\, V_c^2\ ,
\endeq
for some $k>0$ depending on the air density. We refer to \cite{butikov,herrmann,jenkins} for a clear explanation of the relationship between
flutter and internal resonance.\par
The main idea for computing the energy threshold $E_c$ is as follows. The starting point is a nonlinear evolution equation associated to \eq{bridge} such as
\neweq{pde}
u_{tt}+\Delta^2 u+g(u)=0\ ,\qquad\text{for }x\in\Omega\, \ t>0\, ,
\endeq
complemented with the same boundary conditions and some initial data; here, $g(u)$ is a general nonlinearity, possibly nonlocal, describing both the
nonlinear behavior of the structure and the nonlinear action of the sustaining cables. Well-posedness for this problem is proved in \cite{fergaz} by
using the Galerkin method, that is, an approximation of \eq{pde} involving a finite number of modes. By restricting
the attention to a couple of (longitudinal,torsional) modes, problem \eq{pde} is reduced to a second order 2-DOF Hamiltonian system of the form
\neweq{Hamilton}
\ddot{x}+G_x(x,y)=0\ ,\qquad\ddot{y}+G_y(x,y)=0\ ,
\endeq
which has a first integral representing the conserved energy $(\dot{x}^2+\dot{y}^2)/2+G(x,y)$ for some potential $G$.
The stability of \eq{Hamilton} is studied by taking initial data which concentrate almost all the energy on the mode $x$, that is,
$|\dot{y}(0)|+|y(0)|\ll|\dot{x}(0)|+|x(0)|$; this models the situation where the deck is initially oscillating with a wide longitudinal time-dependent
width $x=x(t)$ and with an imperceptible torsional amplitude $y=y(t)$.\par
For a nonlinear string equation, Cazenave and Weissler \cite{cazw} (see also \cite{cazwei}) were able to study the stability for each
couple of modes; this was possible because their nonlinearity was due to a nonlocal term which allowed the dynamics to remain concentrated on the
initially excited modes (through separation of variables) and, hence, they obtained a system such as \eq{Hamilton}. We also refer to \cite{bbfg} for stability results for a nonlinear nonlocal beam equation. A much more difficult case appears to be that involving local terms, although recent
attempts in \cite{bfg} show that, at least with suitable truncations and approximations, this seems to be possible: it was found there that if the total
energy overcomes a threshold $E_c>0$ that can be computed numerically, then there is a transfer of energy from $x$ to $y$, that is, the amplitude of the
torsional component suddenly grows up. This phenomenon is related to the theory of normal forms (see e.g.\ \cite{SVM}) and has sound explanations through
the Floquet theory and Poincar\'e maps (see e.g.\ \cite{chicone}).\par
Since the ``exact and explicit'' form of $g(u)$ in \eq{pde} is not known due to the extreme complexity of the mechanical system, in the next section
we consider some ``toy models'' from which we derive a general rule for \eq{Hamilton}, independently of $G$.

\section{Energy thresholds for some simple Hamiltonian systems}\label{criticalenergy}

We report below on five numerical experiments that we performed. For each experiment, on the left we write the Hamiltonian system considered and the
corresponding conserved energy, whereas on the right we plot the dependence of the energy threshold $E_c$ on the ratio $\gamma$ of the eigenvalues
(or squared frequencies); this plot is obtained by increasing with step $0.1$ the ratio starting from $\gamma=1$ and then by interpolation.\par\bigskip

\begin{minipage}{76mm}
$$(A)\ \left\{\begin{array}{l}
\ddot{x}+(1+x^2+y^2)x=0\\
\ddot{y}+\gamma(1+x^2+y^2)y=0
\end{array}\right.$$
\par\bigskip
$$E=\frac{\dot{x}^2}{2}+\frac{\dot{y}^2}{2\gamma}+\frac{x^2}{2}+\frac{y^2}{2}+\frac{(x^2+y^2)^2}{4}$$
\end{minipage}\qquad
\begin{minipage}{80mm}
{\includegraphics[height=38mm, width=70mm]{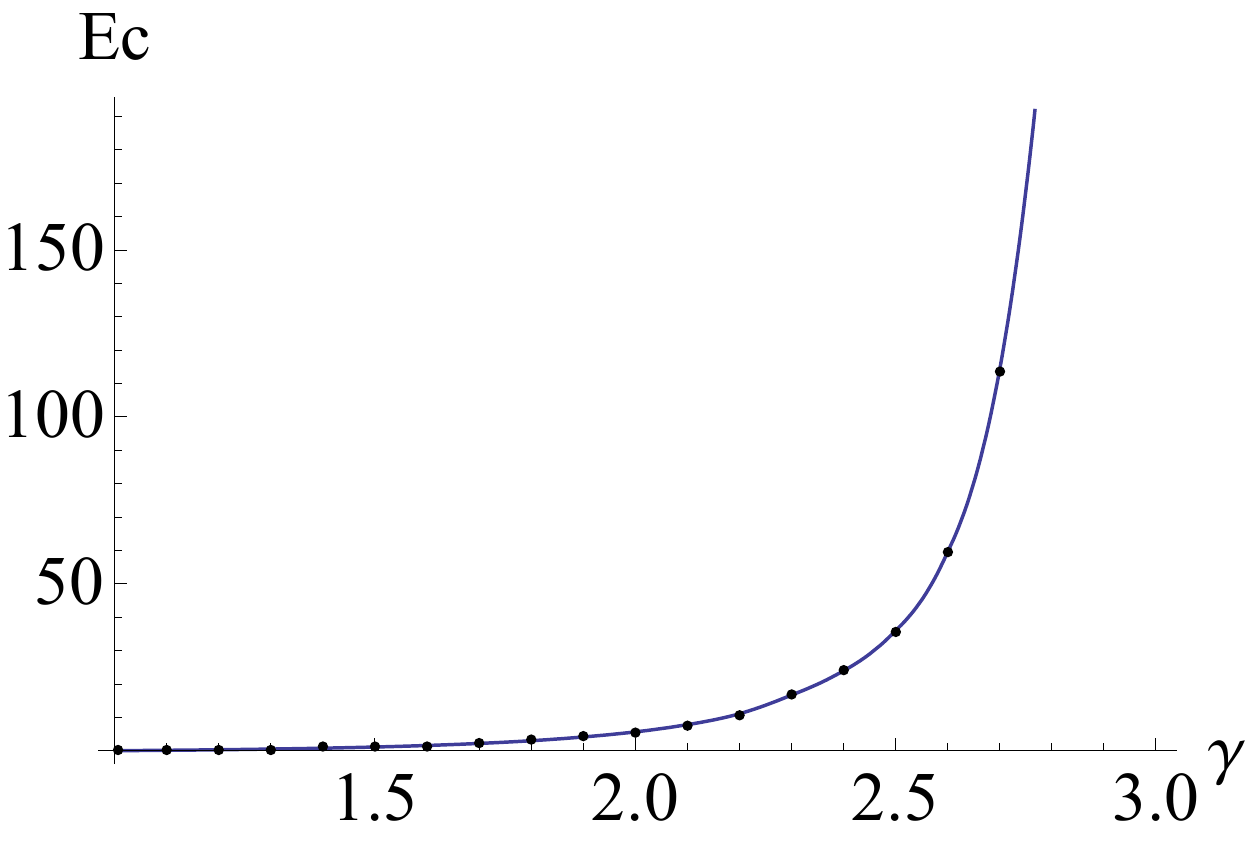}}
\end{minipage}\par\bigskip\noindent
System $(A)$ was derived by Cazenave and Weissler \cite{cazwei} while considering a nonlinear nonlocal wave equation; up to minor changes it
may also be obtained from a beam equation, see \cite{bbfg}. We observe that the energy threshold $E_c$ is a convex and increasing function of $\gamma$
(the squared frequency ratio). It has a behavior similar to $a(\gamma^b-1)^c$ for some $a,b,c>0$
although no choice of these parameters really fits into the displayed plot: in fact, the growth of $E_c(\gamma)$ is more similar to an exponential.
We detected no instability for $(A)$ when $\gamma\le1$; it is proved analytically in \cite{bbfg} that $(A)$ is stable whenever $\gamma<0.955$ although
the conjecture is precisely that $(A)$ is stable for $\gamma\le1$. With simple scalings of $(A)$ one can derive the energy thresholds also for the
systems (with $\alpha,\beta>0$)
$$(A')\ \left\{\begin{array}{l}
\ddot{x}+(\alpha+x^2+y^2)x=0\\
\ddot{y}+\gamma(\alpha+x^2+y^2)y=0\, ,
\end{array}\right.\qquad\qquad
(A'')\ \left\{\begin{array}{l}
\ddot{x}+(1+\beta(x^2+y^2))x=0\\
\ddot{y}+\gamma(1+\beta(x^2+y^2))y=0\, .
\end{array}\right.$$
It is clear that, qualitatively, the plot of $E_c$ remains the same although it is quantitatively different.\par
A slightly different qualitative behavior is exhibited by system $(B)$. This system is also derived from a nonlinear nonlocal
wave equation, see \cite{cazwei,cazw}.

\begin{minipage}{76mm}
$$(B)\ \left\{\begin{array}{l}
\ddot{x}+(1+x^4+y^2)x=0\\
\ddot{y}+\gamma(1+x^2+y^4)y=0
\end{array}\right.$$
\par\bigskip
$$E=\frac{\dot{x}^2}{2}+\frac{\dot{y}^2}{2\gamma}+\frac{x^2}{2}+\frac{y^2}{2}+\frac{x^6}{6}+\frac{y^6}{6}+\frac{x^2y^2}{2}$$
\end{minipage}\qquad
\begin{minipage}{80mm}
{\includegraphics[height=38mm, width=70mm]{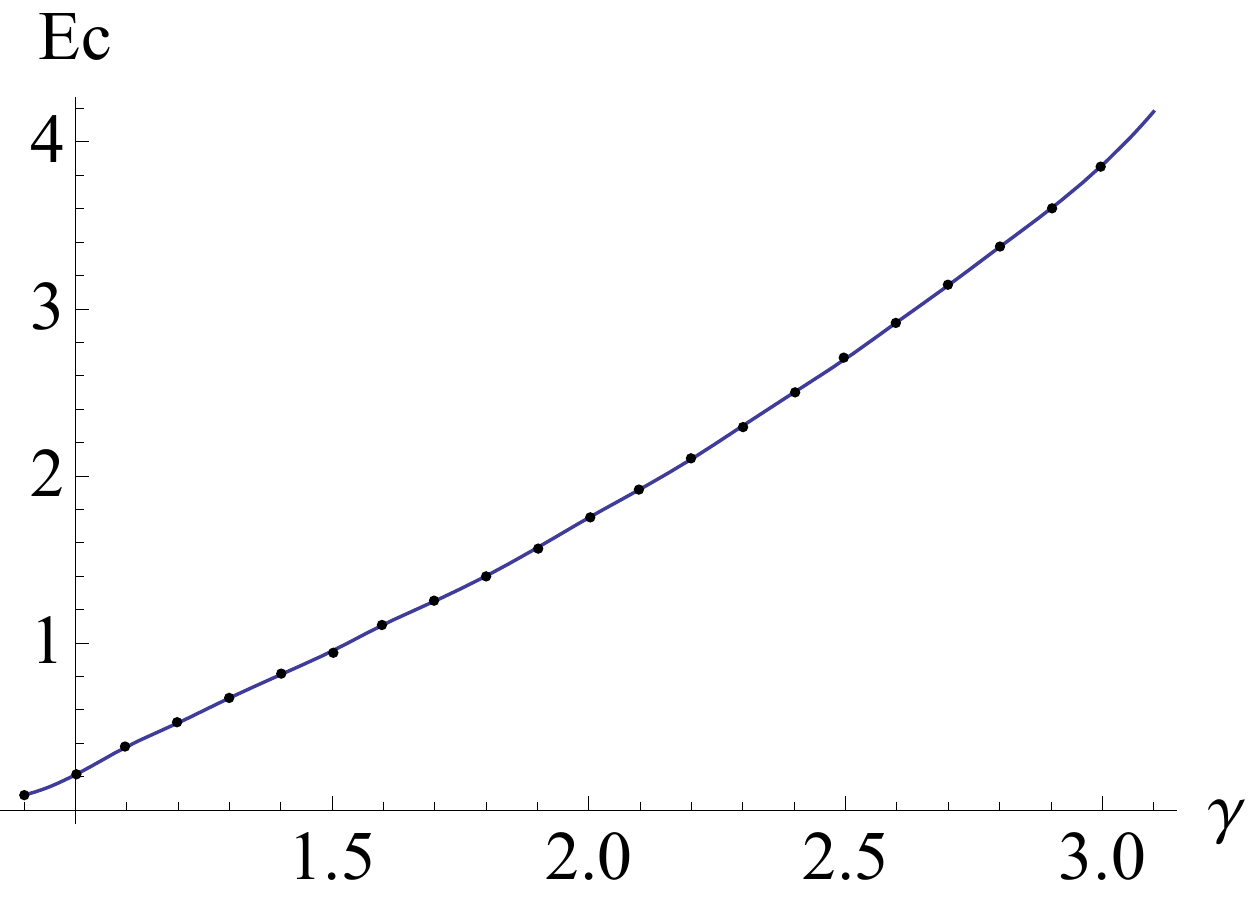}}
\end{minipage}\par\bigskip\noindent
The map $\gamma\mapsto E_c(\gamma)$ is again increasing and convex for $\gamma\ge1$ but there are two crucial differences: we have $E_c(1)>0$
and also $E_c(\gamma)>0$ for all $\gamma>0.84$ while it seems that there is no energy transfer for $\gamma<0.84$. Whence, there may be energy
transfer also from the component with larger frequency towards the component with smaller frequency. We found that $E_c$ exists and is decreasing
for $\gamma\in(0.84,0.98)$. Even if the numerical results exhibit some instability, they are neat and we think that they are reliable.
But we have no explanation of these facts, besides the possible unexpected behavior of resonance tongues in some Hill equations, see
e.g.\ \cite{broer,broer2}.\par
A completely different behavior is displayed by system $(C)$ which was intensively studied in \cite{bgz} in order to figure out a criterion for
the energy transfer in Hamiltonian systems with more than 2-DOF.

\begin{minipage}{76mm}
$$(C)\ \left\{\begin{array}{l}
\ddot{x}+(1+y^2)x=0\\
\ddot{y}+\gamma(1+x^2)y=0
\end{array}\right.$$
\par\bigskip
$$E=\frac{\dot{x}^2}{2}+\frac{\dot{y}^2}{2\gamma}+\frac{x^2}{2}+\frac{y^2}{2}+\frac{x^2y^2}{2}$$
\end{minipage}\qquad
\begin{minipage}{80mm}
{\includegraphics[height=38mm, width=70mm]{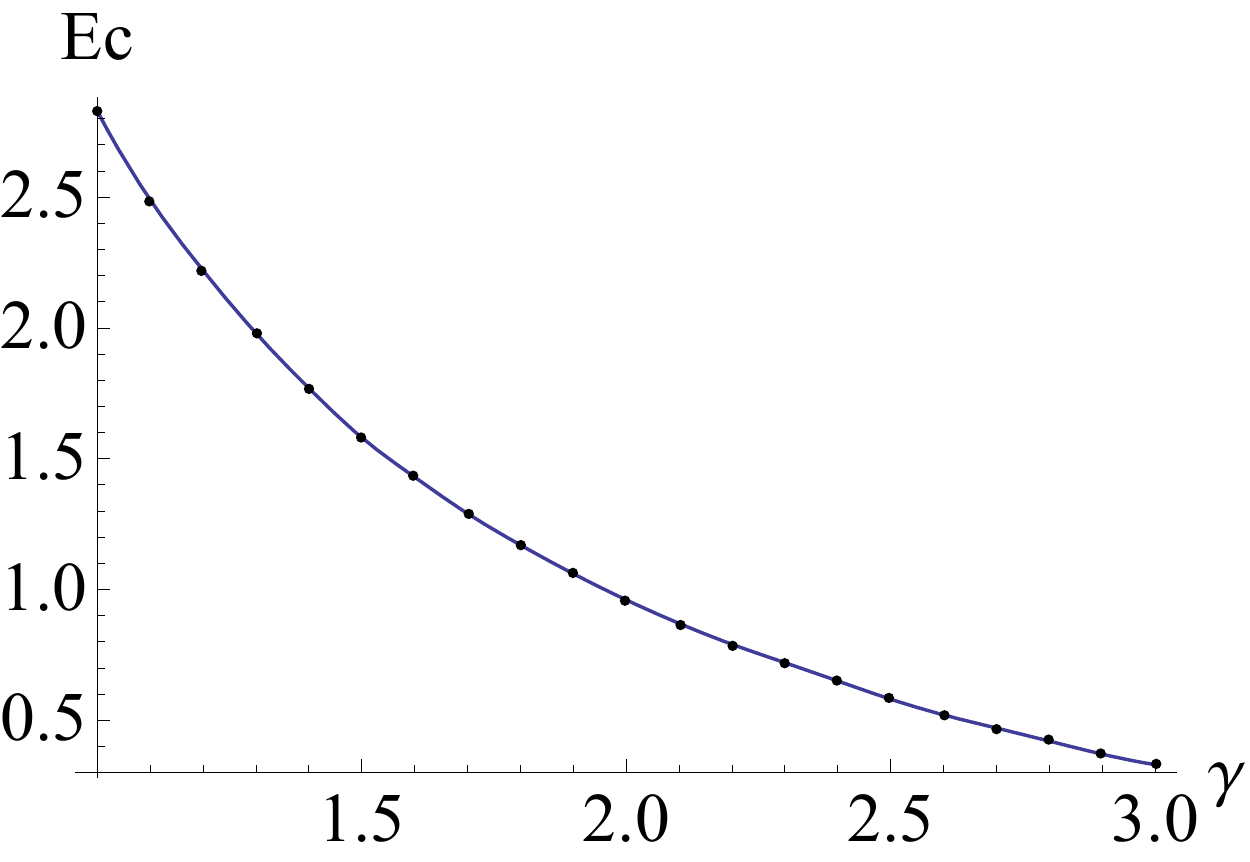}}
\end{minipage}\par\bigskip\noindent
The map $\gamma\mapsto E_c(\gamma)$ is now strictly decreasing although it maintains convexity.
This has a sound explanation in terms of the Mathieu tongues of instability, see \cite{bgz}. Roughly speaking, if $y$ is small then $(C)$ has solutions
close to $(x(t),y(t))\approx(x(0)\cdot\cos t,0)$ when $\dot{x}(0)=0$. By replacing this solution into the second equation, we obtain the Mathieu equation
$\ddot{y}+\gamma(1+x(0)^2\cos^2t)y=0$ for which the resonance tongues are explicitly known, see e.g.\ \cite{cesari,chicone}. By following the Mathieu functions,
one may justify the displayed behavior of $E_c(\gamma)$.\par
The just described results show that the energy thresholds of Hamiltonian systems are very sensitive to the coupling terms.
The energies of the three Hamiltonians $(A)$, $(B)$, and $(C)$, all have the terms
\neweq{postogamma}
\frac{\dot{x}^2}{2}+\frac{\dot{y}^2}{2\gamma}+\frac{x^2}{2}+\frac{y^2}{2}+\frac{x^2y^2}{2}\, .
\endeq
In $(C)$ there is no additional term, in $(A)$ there is the additional term $\frac{x^4}{4}+\frac{y^4}{4}$, whereas in $(B)$ this term is replaced by
$\frac{x^6}{6}+\frac{y^6}{6}$. Still, they exhibit fairly different behaviors.\par
At this point, one may wonder if the ratio of the eigenvalues $\gamma$ is put in the right place in the energy \eq{postogamma}. Therefore we performed
numerical experiments also for the systems $(D)$ and $(E)$. What we discovered is disarming.\par
\begin{minipage}{76mm}
$$(D)\ \left\{\begin{array}{l}
\ddot{x}+(1+x^2+y^2)x=0\\
\ddot{y}+(\gamma+x^2+y^2)y=0\, .
\end{array}\right.$$
\par\bigskip
$$E=\frac{\dot{x}^2}{2}+\frac{\dot{y}^2}{2}+\frac{x^2}{2}+\gamma\frac{y^2}{2}+\frac{(x^2+y^2)^2}{4}$$
\end{minipage}\qquad
\begin{minipage}{80mm}
{\includegraphics[height=38mm, width=70mm]{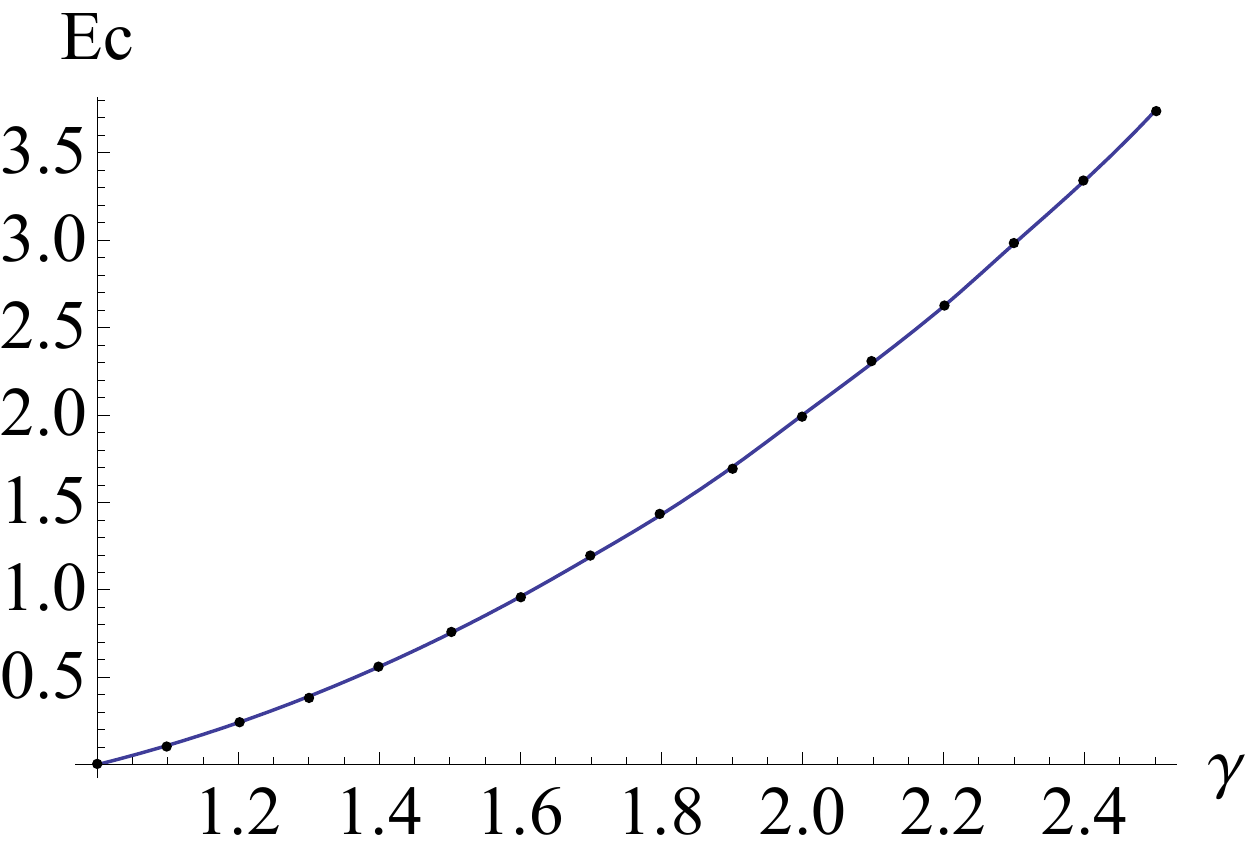}}
\end{minipage}\par\bigskip\noindent
The linearizations of $(D)$ and $(A)$ are the same, that is $\ddot{x}+x=0$ and $\ddot{y}+\gamma y=0$, the ratio of the squared frequencies is
$\gamma$ and the plots of the two critical energies is qualitatively similar (increasing and convex). On the other hand, if we perturb $(D)$ and
consider $(E)$ the energy becomes increasing and concave.\par
\begin{minipage}{76mm}
$$(E)\ \left\{\begin{array}{l}
\ddot{x}+(1+x^2+y^2+2x^6)x=0\\
\ddot{y}+(\gamma+x^2+y^2)y=0\, .
\end{array}\right.$$
\par\bigskip
$$E=\frac{\dot{x}^2}{2}+\frac{\dot{y}^2}{2}+\frac{x^2}{2}+\gamma\frac{y^2}{2}+\frac{(x^2\!+\!y^2)^2}{4}+\frac{x^8}{4}$$
\end{minipage}\qquad
\begin{minipage}{80mm}
{\includegraphics[height=38mm, width=70mm]{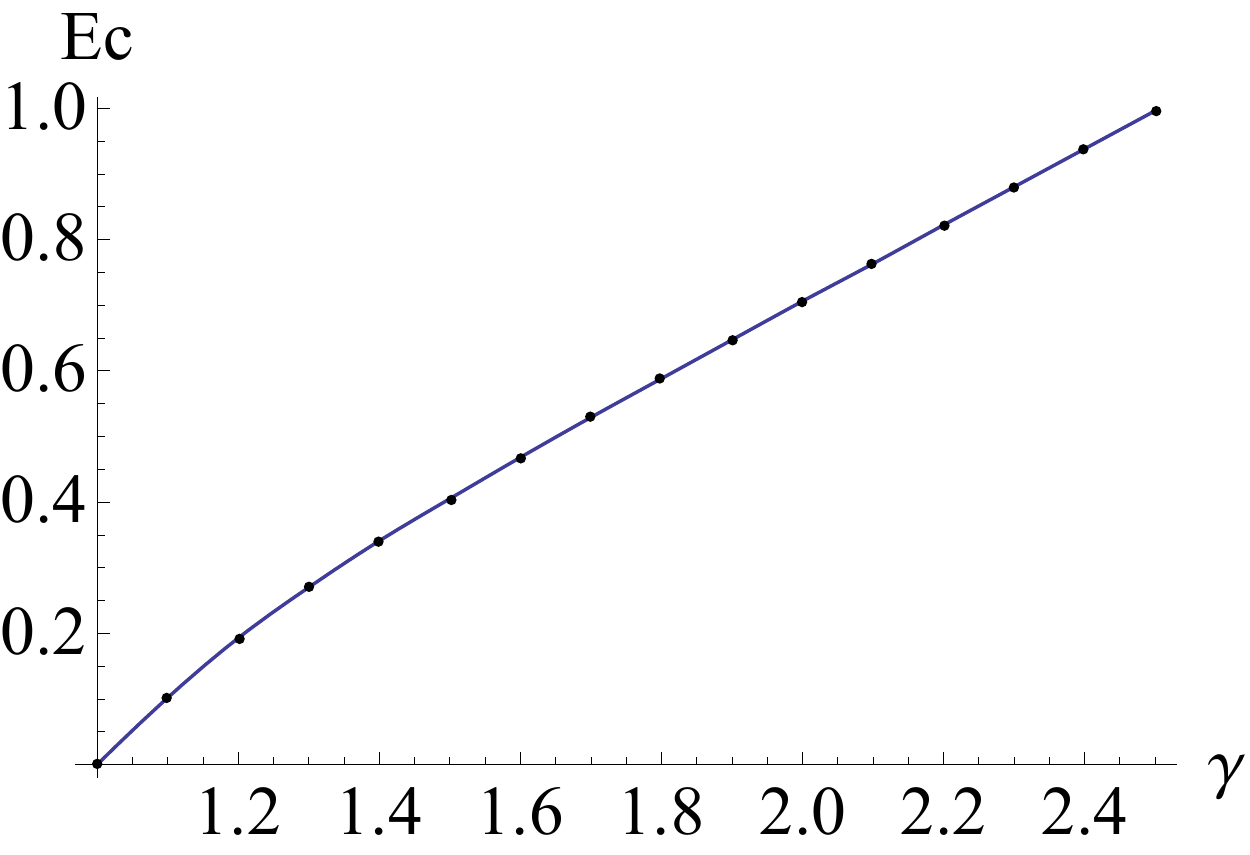}}
\end{minipage}\par\bigskip\noindent

We conclude this section with some remarks aiming to derive a general rule between the considered Hamiltonian systems.
The energy \eq{postogamma} and its perturbations may not satisfy the crucial property that $E_c(1)=0$ whereas the energy of systems $(D)$ and $(E)$ satisfies
this property. On the other hand, $(D)$ and $(E)$ show that
$$
\mbox{there is no common rule on the convexity of the map $\gamma\mapsto E_c(\gamma)$.}
$$
However, all the considered Hamiltonian systems also share a common feature: all the plots displayed in this section show that the map $\gamma\mapsto E_c(\gamma)$ has a ``nice and regular'' graph, representing some simple looking smooth function. These results suggest that\renewcommand{\arraystretch}{1.1}
\neweq{suggest1}
\begin{array}{c}
\mbox{{\bf the energy threshold for the instability of a 2-DOF Hamiltonian system}}\\
\mbox{{\bf depends on the ratio of the squared frequencies.}}
\end{array}\endeq\renewcommand{\arraystretch}{1.5}

\section{Is there a shape functional able to compute the torsional instability?}\label{nostro}

In his pioneering monograph, Rocard \cite[p.164]{rocard} writes that {\em a wide bridge is more stable than a narrow bridge}. For rectangular
plates $\Omega^\ell=(0,\pi)\times(-\ell,\ell)$, this means that
$$
\mbox{{the map}\ }\ell\mapsto V_c(\Omega^\ell)\mbox{ {is strictly increasing}.}
$$
In fact, Rocard \cite[p.186]{rocard} also writes that {\em if all the other factors remain equal and if the natural frequency of bending has
a fixed value, then a bridge twice as wide will have exactly double critical speed $V_c$}. This may be rephrased for rectangular plates as:
\neweq{suggest3}
\mbox{{\bf the map }}\ell\mapsto V_c(\Omega^\ell)\mbox{ {\bf is linearly increasing.}}
\endeq
Condition \eq{suggest3} shows that a functional aiming to compute the flutter velocity of $\Omega^\ell$ should have the
form $V_c=c\, \ell \varphi(\mu,\nu)$ (with $c>0$) for the couple of (longitudinal,torsional) eigenvalues
$(\mu,\nu)$; as we shall see in formula \eq{flutterocard} below, this is precisely the form derived by Rocard. In view of \eq{EcVc}, we may rewrite
this formula in terms of the critical energy of the rectangular plate $\Omega^\ell$, depending on its width:
\neweq{functional}
E_c(\Omega^\ell)=c\, \ell^2\, f(\mu,\nu)\, .
\endeq

So far, we have only considered rectangular plates $\Omega^\ell$. However, in some cases, there is no physical space to enlarge the hinged part of the deck, see
for instance the Aizhai Suspension Bridge \cite{aizhai} and its pictures available on the web; see also the pictures of the cross sections of the
decks of the Severn Bridge and of the Humber Bridge in \cite[Section 2.3.1]{jhnm}. Therefore, one may also be interested in modifying the shape of
the free edges without altering the width of the hinged ones, as in Theorem \ref{duesetted}. Physically meaningful shapes should be symmetric
with respect to the midline of the deck. Therefore, we focus our attention on plates which can be described by
\neweq{newkind}
\Omega_\phi=\big\{(x,y)\in\R^2;\, 0<x<\pi,\, -\ell-\phi(x)<y<\ell+\phi(x)\big\},
\endeq
where $\phi(0)=\phi(\pi)=0$ is a continuous function on $[0,\pi]$ as defined in \eq{phi_def}, see Figure \ref{norectangle}. With this notation we have that $\Omega_{0}=\Omega^\ell$.

\begin{figure}[ht]
\begin{center}
{\includegraphics[height=35mm, width=105mm]{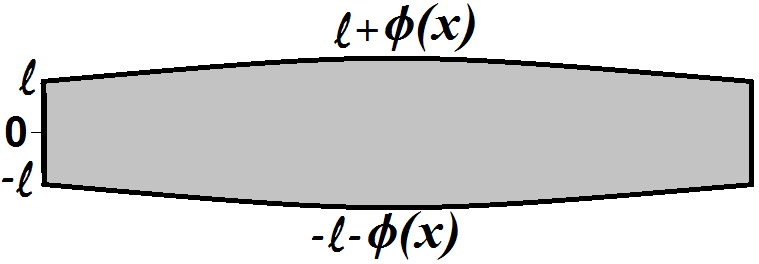}}
\caption{A possible non-rectangular plate $\Omega_\phi$.}\label{norectangle}
\end{center}
\end{figure}

We now generalize \eq{functional} to the case of plates $\Omega_\phi$ of the kind \eq{newkind}. In particular, we must modify the factor $\ell^2$
with a term which also depends on the function $\phi$. The radius of gyration of a plate is related to the moment of inertia with respect to
its axis of symmetry: it is the quadratic mean distance of the parts of the plate from the midline axis (see also Section \ref{flutter}). Therefore, we replace the squared half-width
$\ell^2$ of the rectangular plate $\Omega^\ell$ with the {\em mean value of the squared half-width}:
\neweq{Lphi}
L(\phi):=\frac{1}{\pi}\int_0^\pi\big(\ell+\phi(x)\big)^2\, dx\, .
\endeq
A possible generalization of \eq{functional} able to measure also the flutter velocity of the plate $\Omega_\phi$ in \eq{newkind} then reads
\neweq{flutterocardgen}
E_c(\Omega_\phi)=C\, L(\phi)\, f(\mu,\nu)\,.
\endeq
Using the fact that $\Omega_0=\Omega^\ell$ in \eq{flutterocardgen}, we find \eq{functional}: hence \eqref{flutterocardgen} is indeed a generalization of \eqref{functional}.
Even if Remark \ref{notation_phi} states that variations of $\Omega$ with respect to $\phi=1$ and with respect to $\phi=\sin x$ should not be compared, in view of \eqref{osservazione} and since they both contribute to enlarge the width of the
plate, it is reasonable to expect that
\neweq{reasonable}\renewcommand{\arraystretch}{1.1}
\begin{array}{cc}
\mbox{\bf $E_c$ has the same monotonicity when the long edges of $\Omega^\ell$}\\
\mbox{\bf are perturbed by $\phi=1$ and by $\phi=\sin(x)$.}
\end{array}\renewcommand{\arraystretch}{1.5}
\endeq

In order to understand if the plate increases or decreases its stability, we need to compute the
variation of $E_c$. To this aim, we perturb the free edges
$(0,\pi)\times\{-\ell,\ell\}$ of $\Omega^\ell$ with the function $\phi$ and compute the derivative of $L$ in the direction $\phi$ that we denote by
\neweq{chch}
L'(\phi):=\lim_{\eps\to0}\frac{L(\eps\phi)-L(0)}{\eps}=\frac{2\ell}{\pi}\int_0^\pi \phi(x)\, dx\, .
\endeq
Note that when $\phi= 1$ we get $L'(1)=2\ell$ which is the derivative of $\ell^2$.

\begin{remark} {\em The shape derivative formula \eq{chch} does not change if we replace \eq{Lphi} with
$$\widetilde{L}(\phi):=\left(\frac{1}{\pi}\int_0^\pi\big(\ell+\phi(x)\big)\, dx\right)^2$$
which represents the {\it squared mean value of the half-width}.}\end{remark}

Let us now turn to the function $f=f(\mu,\nu)$. Rocard \cite[p.169]{rocard} claims that for the usual design of bridges the torsional frequency $\omega_t$
is larger than the longitudinal frequency $\omega_v$; further evidence of this fact comes from Irvine \cite[p.178]{irvine}. Moreover, Rocard claims that the
bridge is stable if $\omega_t<\omega_v$ and very unstable (with small $V_c$) if $\omega_t\approx\omega_v$ with $\omega_t>\omega_v$. This results in
the following properties for the function $f$, assumed to be continuous:
$$f(\mu,\nu)>0\quad\forall\nu>\mu>0\ ,\qquad f(\mu,\mu)=0\quad\forall\mu>0\, .$$
Notice that, in view of \eq{comparison}-\eq{comparison2} and \eq{large freq}, we restrict our attention to the following couples of eigenvalues
\begin{eqnarray}
(\bar\mu_{m,1},\bar\nu_{1,2})& & \text{ with }  m=1,...,10\,;\notag\\
(\bar\mu_{m,1},\bar\nu_{n,2})& &\text{ with } m=1,...,14;n=2,...,5\,;\label{choices}\\
(\bar\mu_{m,k},\bar\nu_{n,3})& &\text{ with } k=1,2\,;m=1,...,14;n=1,...,5\,. \notag
\end{eqnarray}
Furthermore, since from \eq{suggest1} we learn that the energy threshold for the instability should depend on the ratio of the squared frequencies,
we end up with the following family of functions
\neweq{suggest4}
f(\mu,\nu)=g\left(\frac{\nu}{\mu}\right)\quad \text{with } g\in C^0[1,+\infty]\,: g(s)>0 \quad\forall s>1\,,\quad  g(1)=0\,.
\endeq
This behavior is qualitatively the one displayed by the Hamiltonian systems $(A)$, $(D)$, and $(E)$ in Section \ref{criticalenergy}, while the systems $(B)$ and $(C)$
fail to satisfy the last condition in \eq{suggest4}. Keeping in mind \eq{suggest4}, by using Corollary \ref{implicitfunctioncorollary} and
Corollary \ref{esplicito}, one may compute the variation of $E_c(\Omega^\ell)$ and $E_c(\Omega_\phi)$ as defined in \eq{functional} and \eq{flutterocardgen}. In the first case, assuming $g$ also differentiable, we obtain
\neweq{diffg}
D_{\ell}E_c(\Omega^\ell)=c\ell \left[2g\left(\frac{\nu}{\mu}\right)+\ell g'\left(\frac{\nu}{\mu}\right)D_{\ell}\left(\frac{\nu}{\mu}\right)\right]\,,
\endeq
and a similar formula may be derived for $E_c(\Omega_\phi)$. Since from Table \ref{tableder} we inferred the empirical rule \eq{odeigen}, formula \eq{diffg}
contradicts \eq{suggest3} for $\nu\approx \mu$ (recall that the corresponding frequencies are $\omega_v=\sqrt{\mu}$ and $\omega_t=\sqrt{\nu}$). Namely,
\neweq{finalc}
\mbox{{\bf there exists no function $f$ such that \eq{suggest3} and \eq{suggest4} hold for all $\nu>\mu>0$}.}
\endeq

Nevertheless, functions in the form \eq{suggest4} for which \eq{suggest3} holds for most of the couples of eigenvalues \eq{choices} can be found. We quote here below two examples of $g$ and, in turn, of $f$.\par\smallskip\noindent
$\bullet$ $g(s)=(s-1)^{1/10}+\frac{1}{10}(s-1)$ gives positive numbers in \eq{diffg} for all couples in $\eq{choices}$ except for $(\bar\mu_{m,1},\bar\nu_{n,3})$ with $m=1,...,14$ and $n=1,...,5$.\par
If we consider instead $D_1 E_c(\Omega_{\phi})$, we have that: it is negative for the couples $(\bar\mu_{m,1},\bar\nu_{1,2})$ with $m=1,...,10$ and for the couples $(\bar\mu_{m,1},\bar\nu_{n,3})$ with $m=1,...,14$ and $n=1,...,5$; positive for $(\bar\mu_{m,2},\bar\nu_{n,3})$ with $m=1,...,14$ and $n=1,...,5$. For $(\bar\mu_{m,1},\bar\nu_{n,2})$ we get both positive and negative derivatives depending on $m$ and $n\geq 2$.
\par\smallskip\noindent
$\bullet$ $g(s)=(s-1)^{1/4}$ gives positive numbers in \eq{diffg} for all couples in $\eq{choices}$ except for $(\bar\mu_{10,1},\bar\nu_{1,2})$ and $(\bar\mu_{14,1},\bar\nu_{2,2})$.
\par
If we consider instead $D_1 E_c(\Omega_{\phi})$, we have that: it is negative for the couples $(\bar\mu_{m,1},\bar\nu_{n,3})$ with $m=1,...,14$
and $n=1,...,5$; positive for the couples $(\bar\mu_{m,2},\bar\nu_{n,3})$ with $m=1,...,14$ and $n=1,...,5$. For $(\bar\mu_{m,1},\bar\nu_{n,2})$ we get both positive and negative derivatives depending on $m$ and $n$.
\par\smallskip
We observe that these functions violate both \eq{suggest3} and \eq{reasonable} for some couples of eigenvalues. This gives additional evidence to the statement \eqref{finalc}.

\section{A quick overview of the engineering literature}\label{flutter}

Let $M$ be the mass of the unit length of the deck (steel and concrete assembled within the same unit length) and let $M_0$ be the mass of air in a
square parallelepiped erected above unit length. For common bridges the ratio $M/M_0$ is around 50, see \cite[p.151]{rocard}. Rocard \cite[p.169]{rocard} considers the natural longitudinal and torsional frequency of the bridge denoted, respectively, by $\omega_v$
and $\omega_t$: he claims that for common bridges one has $\omega_t>\omega_v$ and that the critical velocity of the wind for which the
bridge undergoes to flutter may be computed through the formula \cite[p.163]{rocard}
\neweq{flutterocard2}
V_c^2=\frac{2r^2\ell^2}{2r^2+\ell^2}\frac{M}{M_0}(\omega^2_t-\omega^2_v)\, ,
\endeq
where $2\ell$ is the width of the deck and $r$ is the radius of gyration of the unit length in the deck.
Rocard \cite[p.142]{rocard} well explains the role of the radius of gyration and shows that $r\approx\ell/\sqrt2$. Therefore, formula
\eq{flutterocard2} becomes
\neweq{flutterocard}
V_c^2(\Omega^\ell)=C\, \ell^2(\omega^2_t-\omega^2_v)\ ,
\endeq
where $C>0$ is a physical constant depending on the shape and the material composing the deck.\par
Formula \eq{flutterocard2} was later modified by Selberg \cite{selberg} and, more recently, in \cite[Formula (20)]{matsumoto}; in both cases, there is a
different multiplicative constant $C$ in front of $\ell^2(\omega^2_t-\omega^2_v)$. Let us also mention that related formulas were suggested by
Irvine \cite[Formula (4.91)]{irvine} and by Como, Del Ferraro, and Grimaldi \cite[$\S$ 8]{como}. All these references share the idea of the dependence of
the stability on the quantity $\ell^2(\omega^2_t-\omega^2_v)$. We refer to \cite[$\S$ 1.7.4]{bookgaz} for more details and references.\par
Summarizing, there is an evident disagreement on the exact value of the constant $C>0$, the reason being that in most cases the value of the ``characteristic
parameters'' of a bridge are not found theoretically but with wind tunnel tests, which allow to compute the so-called flutter derivatives also known as aeroelastic (or aerodynamic) derivatives. These are the coefficients
to be inserted in suitable linear ODE's used to find the longitudinal and torsional frequencies, as first suggested by Scanlan and Tomko \cite{scantom}. The flutter
speed may then be computed through closed formulas, see e.g.\ \cite[Formula (9.156)]{lacarbonara}, which, however, are very different from
\eq{flutterocard}.\par
We tested \eq{flutterocard} numerically: by \eq{EcVc}, this formula also reads
$$E_c(\Omega^\ell)=c\, \ell^2\, f(\mu,\nu)\qquad\mbox{with}\quad f(\mu,\nu)=\max\{\nu-\mu,0\}\, .$$
Again, in view of \eq{comparison}-\eq{comparison2}, we focused our attention only on the three families of couples in \eq{choices}.
By using Theorems \ref{implicitfunction} and \ref{duesetted} as well as their corollaries, we obtained the following results.\par
$\bullet$ When $\phi=1$, the derivatives of $E_c$ are negative for all the couples in \eqref{choices} except for $(\bar\mu_{m,1},\bar\nu_{n,2})$ with $n=2,....,5$ and $m=1,...,n-1$.\par
$\bullet$ When $\phi=\sin x$, the derivatives of $E_c$ are negative for all the couples in \eqref{choices}.\par
$\bullet$ When $\phi=3\sin 3x$, the derivatives of $E_c$ are negative for the couples in \eqref{choices} with $n\neq1$.\par
From Remark \ref{notation_phi} we recall that the case $\phi=1$ and the cases $\phi=h\sin hx$ are not directly comparable, since the former is broadening the hinged edges, while the latter is just modifying the free edges. However, these results show that the functional \eq{flutterocard}
and its extension \eq{flutterocardgen} fail to satisfy \eq{reasonable}. Moreover, the results for $\phi=1$ show that \eq{flutterocard} fails to
satisfy \eq{suggest3}. Overall, this suggests to consider \eq{flutterocard} unreliable. Let us emphasize that also the engineering literature
is quite skeptic about \eq{flutterocard}; in particular, Holmes \cite[p.293]{holmes} shows that the Selberg's formula does not always agree with
experimental measurements.

\section{Proof of Theorem \ref{eigenvalue}}\label{proof1}

Theorem \ref{eigenvalue} is essentially proved in \cite{bfg,fergaz}, except for two improvements on the bounds for the eigenvalues.\par
First, by Proposition \ref{defunzioni} we know that the least longitudinal eigenvalue $\mu_{m,1}$ corresponding to eigenfunctions with $(m-1)$ nodes in the
$x$-direction is known to be the unique value $\lambda\in((1-\sigma)^2m^4,m^4)$ such that $\Phi^m(\lambda,\ell)=0$.
In Theorem \ref{eigenvalue} we have a larger lower bound which we now prove.

\begin{lemma} For any $m\geq 1$, there holds
\neweq{lowerbound}
\mu_{m,1}>(1-\sigma^2)m^4.
\endeq
\end{lemma}
\begin{proof} Note that $\Phi^m((1-\sigma)^2m^4,\ell)>0$ and $\Phi^m(m^4,\ell)<0$; whence,
\eq{lowerbound} follows if we show that $\Phi^m((1-\sigma^2)m^4,\ell)>0$. After division by $m^5$, after computing the two squared parenthesis $(\cdot)^2$,
and after a further division by $2(1-\sigma)$, we see that the sign of $\Phi^m((1-\sigma^2)m^4,\ell)$ is the same as
$$
\frac{\tanh(\ell m\sqrt{1-(1-\sigma^2)^{1/2}})}{\sqrt{1-(1-\sigma^2)^{1/2}}}-
\frac{\tanh(\ell m\sqrt{1+(1-\sigma^2)^{1/2}})}{\sqrt{1+(1-\sigma^2)^{1/2}}}\, .
$$
But the sign of this term is positive since the map $s\mapsto\frac{\tanh s}{s}$ is strictly decreasing. This proves \eq{lowerbound}.\end{proof}

Next, the (new) bounds for $\mu_{m,k}$ are obtained as follows.

\begin{lemma}
For any $m\geq 1$ and $k\geq 2$, there holds
$$
\left(m^2+\frac{\pi^2}{\ell^2}\left(k-\frac{3}{2}\right)^2\right)^2<\mu_{m,k}<\left(m^2+\frac{\pi^2}{\ell^2}\left(k-1\right)^2\right)^2.
$$
\end{lemma}
\begin{proof}
By definition and by Proposition \ref{defunzioni}, for $m$ fixed, $\mu_{m,k}$ is the $(k-1)-$th positive solution $\lambda$ of the equation $\Upsilon^m(\lambda,\ell)=0$.
It is readily seen that the map $\lambda\mapsto\Upsilon^m(\lambda,\ell)$ is continuous, strictly increasing and goes from $-\infty$ to a positive value in any interval $\left(\left(m^2+\frac{\pi^2}{\ell^2}\left(k-\frac{3}{2}\right)^2\right)^2\,,\left(m^2+\frac{\pi^2}{\ell^2}\left(k-1\right)^2\right)^2\right)$,
from which the thesis follows.
\end{proof}

\section{Proof of Theorem \ref{implicitfunction}}

For $\ell=\pi/150$, Proposition \ref{defunzioni} states that the longitudinal eigenvalue $\bar\mu_{m,1}$ is the unique value of $\lambda\in(0.96m^4,m^4)$ such that $\Phi^m(\bar\mu_{m,1},\frac{\pi}{150})=0$ and the torsional eigenvalue $\bar\nu_{n,1}$ is the unique value of $\lambda\in(0.96n^4,n^4)$ such that $\Gamma^n(\bar\nu_{n,1},\frac{\pi}{150})=0$. By the Implicit Function Theorem, the relation $\Phi^m(\lambda,\ell)=0$ implicitly defines, in a neighborhood $U$ of $\ell=\frac{\pi}{150}$, a
smooth function $\mu_{m,1}=\mu_{m,1}(\ell)$ such that
$$\mu_{m,1}\left(\frac{\pi}{150}\right)=\bar\mu_{m,1}\, ,\qquad\Phi^m\big(\mu_{m,1}(\ell),\ell\big)=0\quad\forall\ell\in U\, .$$
Similarly, the relation $\Gamma^n(\lambda,\ell)=0$ implicitly defines, in a
neighborhood $V$ of $\ell=\frac{\pi}{150}$, a smooth function $\nu_{n,1}=\nu_{n,1}(\ell)$ such that
$$\nu_{n,1}\left(\frac{\pi}{150}\right)=\bar \nu_{n,1}\, ,\qquad\Gamma^n\big(\nu_{n,1}(\ell),\ell\big)=0\quad\forall\ell\in V\, .$$

In particular, if we denote by $\Phi^m_\ell$ and $\Phi^m_\lambda$ the partial derivatives of $\Phi^m$ and by $\Gamma^n_\ell$ and $\Gamma^n_\lambda$ the partial derivatives of $\Gamma^n$, we have that
$$
\frac{d\mu_{m,1}}{d\ell}\left(\frac{\pi}{150}\right)=-\frac{\Phi^m_\ell(\bar\mu_{m,1},\frac{\pi}{150})}{\Phi^m_\lambda(\bar\mu_{m,1},\frac{\pi}{150})}
\qquad\mbox{and}\qquad
\frac{d \nu_{n,1}}{d\ell}\left(\frac{\pi}{150}\right)=-\frac{\Gamma^n_\ell(\bar\nu_{n,1},\frac{\pi}{150})}{\Gamma^n_\lambda(\bar\nu_{n,1},\frac{\pi}{150})}\,.
$$
Hence,
$$\frac{d}{d \ell}\left(f(\mu_{m,1}(\ell),\nu_{n,1}(\ell))\right)(\tfrac{\pi}{150})=
-f_{\mu}(\bar\mu_{m,1},\bar\nu_{n,1})\frac{\Phi^m_\ell(\bar\mu_{m,1},\frac{\pi}{150})}{\Phi^m_\lambda(\bar\mu_{m,1},\frac{\pi}{150})}-
f_{\nu}(\bar\mu_{m,1},\bar\nu_{n,1})\frac{\Gamma^n_\ell(\bar\nu_{n,1},\frac{\pi}{150})}{\Gamma^n_\lambda(\bar\nu_{n,1},\frac{\pi}{150})},$$

This gives the explicit form of the derivative of $f(\mu_{m,1}(\ell),\nu_{n,1}(\ell))$. The proofs for the other couples follow similarly.

\section{Proof of Theorem \ref{duesetted}}\label{dimostrazione}

We consider the operator $P$ from $H^2_*(\xi(\Omega))$ to its dual, which takes $u\in H^2_*(\xi(\Omega))$ to $P[u]$, implicitly defined by (we omit
the duality crochet $\langle\cdot,\cdot\rangle$ in order to avoid heavy notations)
\begin{eqnarray*}
P[u][v] &=& \int_{\phi(\Omega)} (1-\sigma)D^2u:D^2v+\sigma\Delta u\Delta vdA\qquad\forall v \in H^2_*(\xi(\Omega))\\
 &=:& (1-\sigma)M(u,v)+\sigma N(u,v).
\end{eqnarray*}
The operator $P$ is easily seen to be a linear homeomorphism of $H^2_*(\xi(\Omega))$ onto its dual.
We also denote by ${\mathcal J}$ the continuous embedding of $H^2_*(\xi(\Omega))$ into its dual, implicitly defined by
\begin{equation*}
{\mathcal J}[u][v]:=\int_{\xi(\Omega)} uvdA\qquad \forall v\in H^2_*(\xi(\Omega)).
\end{equation*}
Then \eq{bridge2} can be written in the weak form
\begin{equation}\label{dirweak3}
P[u][v]=\lambda\mathcal{J}[u][v]\qquad\forall v\in H^2_*(\xi(\Omega)).
\end{equation}
We define the operator $T:=P^{-1}\circ {\mathcal J}$ from $H^2_*(\xi(\Omega))$ to itself. We have the following

\begin{lemma}\label{comsadir1}
Let $\phi\in\mathcal{A}_{\Omega}$. The operator $T$ is a non-negative compact selfadjoint operator in the Hilbert space $H^2_*(\xi(\Omega))$. Its spectrum is discrete
and consists of a decreasing sequence of positive eigenvalues of finite multiplicity converging to zero. Moreover, the equation $Tu=(\lambda^{-1}) u$ is satisfied for some
$u\in H^2_*(\xi(\Omega))$  if and only if equation (\ref{dirweak3}) is satisfied for any $v \in H^2_*(\xi(\Omega))$.
\end{lemma}
\begin{proof} For the selfadjointness, it suffices to observe that
$$
\langle  T u,v\rangle =\langle (P^{-1}\circ {\mathcal J}) [u],v\rangle =P\left[(P^{-1}\circ {\mathcal J}) [u]\right][v]={\mathcal J}[u][v],
$$
for any $u,v\in {H}^2_*(\xi(\Omega))$. For the compactness, just observe that the operator $\mathcal J$ is compact. The remaining statements are straightforward.
\end{proof}

In order to prove Theorem \ref{duesetted} we fix $\xi\in{\mathcal A}_\lambda$, consider equation (\ref{dirweak3}) in $\xi(\Omega)$, and pull it back to $\Omega$.
The pull-back $M_\xi$ to $\Omega $ of the operator $M$ on $\xi (\Omega)$ is defined by
\begin{equation*}
M_{\xi }[u][v]=     \int_{\Omega }\left(D^2( u\circ\xi^{-1}):D^2( v\circ\xi^{-1}) \right)\circ\xi|{\rm det}\nabla \xi|dA,
\end{equation*}
for all $u,v  \in H^2_*(\Omega )$, and similarly for $N_{\xi}$ and $P_{\xi}=(1-\sigma)M_{\xi}+\sigma N_{\xi}$. We also note that
$$
\mathcal J_{\xi}[u][v]=\int_{\Omega}uv|{\rm det}\nabla\xi|dA\qquad \forall u,v\in H^2_*(\Omega)\\,
$$
and that the map from $H^2_*(\Omega )$ to  $H^2_*(\xi (\Omega ))$ which maps $u$ to $u\circ \xi^{-1} $ for all $u\in H^2_*(\Omega )$ is a linear homeomorphism.
Hence, equation (\ref{dirweak3}) is equivalent to
\begin{equation*}
P_{\xi }[u][\varphi ]=\lambda \mathcal J_{\xi}[u][\varphi ]\qquad\forall \varphi \in H^2_*(\Omega )\,,
\end{equation*}
where $u=v\circ\xi$. It turns out that the operator $T$ defined in Lemma \ref{comsadir1} is unitarily equivalent to the operator $T_{\xi }$ defined on
$H^2_*(\Omega )$  by
\begin{equation}\label{tphid}
T_{\xi}:=P_{\xi}^{-1}\circ {\mathcal J}_{\xi}.
\end{equation}

We endow the space $H^2_*(\Omega)$ also with the bilinear form
\neweq{endow2}
\langle u,v\rangle _{\xi}=P_{\xi}[u][v] \qquad \forall u,v\in H^2_*(\Omega).
\endeq
We have the following lemma where ${\mathcal{L}}(H^2_{*}(\Omega ))$  denotes the space of linear bounded operators from $H^2_{*}(\Omega )$ to itself and and ${\mathcal{B}}_s(H^2_*(\Omega ))$ denotes the space of bilinear forms on $H^2_*(\Omega )$ (both spaces are equipped with their usual norms).

\begin{lemma}\label{changelemmad}
The operator $T_{\xi}$ defined in (\ref{tphid}) is  non-negative selfadjoint and compact  on the Hilbert space $H^2_*(\Omega)$ endowed with \eq{endow2}.
The equation (\ref{dirweak3}) is satisfied for some $v\in H^{2 }_{*}(\xi(\Omega ))$ if and only if  the equation $T_{\xi } u=(\lambda^{-1}) u$ is satisfied with
$u=v\circ \xi $. Moreover, the map from $\mathcal{A}_{\Omega}$ to ${\mathcal{L}}(H^2_*(\Omega ))\times {\mathcal{B}}_s(H^2_*(\Omega ))$  which takes  $\xi \in \mathcal{A}_{\Omega}$ to $(T_{\xi }, \langle \cdot ,\cdot\rangle _{\xi})$ is real-analytic.
\end{lemma}
\begin{proof}
See \cite[Lemma 3.2]{bula2013} and also \cite{buosohinged}.
\end{proof}

We also need the next statement whose proof can be done following step-by-step those of \cite[Lemmas 2.4 and 2.5]{buosotesi} (see also \cite[Lemma 7]{buosopalinuro} and \cite[Lemma 4.4]{bupro2015}).

\begin{lemma}\label{fondamentale}
Let $\theta\in C^2_b(\Omega;\mathbb{R}^2)$ be of the form $\theta(x,y)=(0,\tau(x)+y\delta(x))$, where $\tau,\delta$ are as in (\ref{psi}). Then for all $w_1,w_2\in H^4(\Omega)$ we have
\begin{multline}
\label{formulafondamentale}
d|_{\xi={\rm Id}}M_{\xi}[w_1][w_2][\theta]=\int_{\partial\Omega}(D^2w_1:D^2w_2)\theta\cdot\nu d\mathcal{H}^1\\
+\int_{\partial\Omega}\left(\mathrm{div}_{\partial\Omega}(\nu\cdot D^2w_1)_{\partial\Omega}\nabla w_2+
\mathrm{div}_{\partial\Omega}(\nu\cdot D^2w_2)_{\partial\Omega}\nabla w_1\right)\cdot\theta d\mathcal{H}^1\\
+\int_{\partial\Omega}\left(\frac{\partial\Delta w_1}{\partial\nu}\nabla w_2+\frac{\partial\Delta w_2}{\partial\nu}\nabla w_1\right)\cdot\theta d\mathcal{H}^1
-\int_{\Omega}\left(\Delta^2w_1\nabla w_2+\Delta^2w_2\nabla w_1\right)\cdot\theta d\mathcal{H}^1\\
-\int_{\partial\Omega}\left(\frac{\partial^2w_1}{\partial\nu^2}\nabla w_2+\frac{\partial^2w_2}{\partial\nu^2}\nabla w_1\right)\cdot\frac{\partial\theta}{\partial\nu}d\mathcal{H}^1
-\int_{\partial\Omega}\left(\frac{\partial^2w_1}{\partial\nu^2}\frac{\partial\ }{\partial\nu}\nabla w_2+\frac{\partial^2w_2}{\partial\nu^2}\frac{\partial\ }{\partial\nu}\nabla w_1\right)\cdot\theta d\mathcal{H}^1,
\end{multline}
and
\begin{multline}\label{formulafondamentale2}
d|_{\xi={\rm Id}}N_{\xi}[w_1][w_2][\theta]=
\int_{\partial\Omega}\Delta w_1\Delta w_2\theta\cdot\nu d\mathcal{H}^1\\
+\int_{\partial\Omega}\left(\frac{\partial\Delta w_1}{\partial\nu}\nabla w_2+\frac{\partial\Delta w_2}{\partial\nu}\nabla w_1\right)\cdot\theta d\mathcal{H}^1
-\int_{\Omega}\left(\Delta^2w_1\nabla w_2+\Delta^2w_2\nabla w_1\right)\cdot\theta d\mathcal{H}^1\\
-\int_{\partial\Omega}\left(\Delta w_1\nabla w_2+\Delta w_2\nabla w_1\right)\cdot\frac{\partial\theta}{\partial\nu}d\mathcal{H}^1
-\int_{\partial\Omega}\left(\Delta w_1\frac{\partial\ }{\partial\nu}\nabla w_2+\Delta w_2\frac{\partial\ }{\partial\nu}\nabla w_1\right)\cdot\theta d\mathcal{H}^1.
\end{multline}
\end{lemma}

The proof of these results in \cite{buosotesi} is lengthy and delicate, covering a number of pages. There, the domain $\Omega$ is assumed to be of class
at least $C^4$, while in Lemma \ref{fondamentale} it is only piecewise smooth. This translates in the appearance of some corner-terms whenever the
Tangential Divergence Theorem is used (namely, at pages 21, 23, 24, 27 in \cite{buosotesi}). However, those terms happen to vanish due to the
particular form of $\theta$, thereby yielding formulas (\ref{formulafondamentale}) and (\ref{formulafondamentale2}).

\begin{proof}[Proof of Theorem \ref{duesetted}.]
By Lemma~\ref{changelemmad}, $T_{\xi }$
is selfadjoint with respect to the scalar product \eq{endow2} and both $T_{\xi }$ and $\langle \cdot ,\cdot \rangle _{\xi} $ depend real-analytically
on $\xi$. Thus, by applying \cite[Theorem 2.30]{lala2004}, it follows that ${\mathcal{A}}_\lambda$ is an open set in $C^2_b(\Omega;{\mathbb{R}}^2)$
and the map which takes $\xi\in {\mathcal { A}}_\lambda$ to $\lambda[\xi(\Omega)]^{-1}$ is real-analytic and therefore also $\xi\mapsto\lambda[\xi(\Omega)]$
is real-analytic (see also \cite[Theorem 3.21]{lala2004}).

It remains to  prove formula (\ref{derivd3}). Let $v_\lambda$ be as in the statement so that $v_\lambda\in H^4(\Omega)$. By
\cite[Theorem 2.30]{lala2004}, it follows that
\begin{equation*}
{\rm d|}_{\xi ={\rm Id}}\lambda[\xi(\Omega)]^{-1}[\zeta]=\langle {\rm d|}_{\xi ={\rm Id}}T_{\xi }[v_\lambda][\zeta], v_\lambda\rangle
\end{equation*}
for all  $\zeta\in C^2_b( \Omega\, ;{\mathbb{R}}^2) $.
We have
$$\langle {\rm d|}_{\xi ={\rm Id} } T_{\xi }[v_\lambda][\zeta], v_\lambda\rangle =d|_{\xi={\rm Id}}\mathcal{J}_{\xi}[v_\lambda][v_\lambda][\zeta]
-\lambda[\Omega]^{-1}d|_{\xi={\rm Id}}P_{\xi}[v_\lambda][v_\lambda][\zeta].$$

Moreover, by standard calculus, ${\rm d}|_{\xi ={\rm Id}} \left({\mathrm{det}} \nabla\xi\right)[\zeta]={\mathrm{div}}\zeta$, and therefore
\begin{equation*}
d|_{\xi={\rm Id}}\mathcal{J}_{\xi}[v_\lambda][v_\lambda][\zeta]=\int_{\Omega}v_\lambda^2\mathrm{div}\zeta dA.
\end{equation*}
Using (\ref{equivalent}), Lemma \ref{fondamentale}, the fact that $(1-\sigma)\frac{\partial^2 v_\lambda}{\partial\nu^2}+\sigma\Delta v_\lambda=0$
on $\partial\Omega$, observing that
\begin{equation*}
\int_{\Omega}\nabla(v_\lambda^2)\cdot\zeta dA=\int_{\partial\Omega}v_\lambda^2\zeta\cdot\nu d\mathcal{H}^1-\int_{\Omega}v_\lambda^2\mathrm{div}\zeta dA,
\end{equation*}
and taking $\zeta=\psi-{\rm Id}$ with $\psi$ in the form (\ref{psi}), we get
\begin{eqnarray*}
d|_{\phi={\rm Id}}\lambda[\phi(\Omega)][\psi-{\rm Id}]
&=& -\lambda[\Omega]\int_0^\pi\left((1-\sigma)|D^2v|^2+\sigma(\Delta v)^2-\lambda[\Omega]v^2\right)|_{y=-\ell}\left(\tau(x)-\ell\delta(x)\right) dx \label{derivd4}\\
 & &+\lambda[\Omega]\int_0^\pi\left((1-\sigma)|D^2v|^2+\sigma(\Delta v)^2-\lambda[\Omega]v^2\right)|_{y=\ell}\left(\tau(x)+\ell\delta(x)\right)dx.\notag
\end{eqnarray*}
Now we observe that all the eigenfunctions show symmetry properties, i.e., they are either even or odd in the $y$ variable, and symmetric or skew-symmetric
with respect to $x$, see Theorem \ref{eigenvalue}. In particular this implies that $|D^2 v|^2$, $(\Delta v)^2$ and $v^2$ are equal for $y=\pm\ell$ and symmetric with respect to $x=\pi/2$.
This proves formula (\ref{derivd3}) and completes the proof of Theorem \ref{duesetted}.
\end{proof}

\section{Some technical lemmas}\label{lemmass}

In this section we quote some results that were used for the numerical computations of the derivatives of the eigenvalues in Table \ref{tableder}.

\begin{lemma}\label{calcoli}
Let $v_{m,k}$ be the eigenfunction associated with the eigenvalue $\mu_{m,k}$ as in Theorem \ref{eigenvalue}. Then
$$
(1-\sigma)|D^2v_{m,k}|^2+\sigma(\Delta v_{m,k})^2-\mu_{m,k}v^2_{m,k}|_{y=\ell}=A_{m,k}\sin^2(mx)+B_{m,k}\cos^2(mx),
$$
where
$$
A_{m,k}=4\mu_{m,k}(m^4-\sigma^2m^4-\mu_{m,k})
$$
and
$$
B_{m,k}=4(1-\sigma)m^2\mu_{m,k}(\mu_{m,k}^{1/2}+m^2)\left(\frac{\mu_{m,k}^{1/2}-(1-\sigma)m^2}{\mu_{m,k}^{1/2}+(1-\sigma)m^2}\right)^2
\tanh^2\left(\ell\sqrt{m^2+\mu_{m,k}^{1/2}}\right).
$$
\end{lemma}
\begin{proof} We have
$$
v_{m,k}|_{y=\ell}=2\mu_{m,k}^{1/2}\sin(mx),\qquad (v_{m,k})_{xx}|_{y=\ell}=-2m^2\mu_{m,k}^{1/2}\sin(mx),
$$
\begin{eqnarray*}
(v_{m,k})_{yy}|_{y=\ell} &=& \left((\mu_{m,k}^{1/2}-(1-\sigma)m^2)(\mu_{m,k}^{1/2}+m^2)
+(\mu_{m,k}^{1/2}+(1-\sigma)m^2)(m^2-\mu_{m,k}^{1/2}) \right)\sin(mx)\\
 &=& 2m^2\sigma\mu_{m,k}^{1/2}\sin(mx),
\end{eqnarray*}
and, if $k\ge2$,
\begin{multline*}
(v_{m,k})_{xy}|_{y=\ell}=\left((\mu_{m,k}^{1/2}-(1-\sigma)m^2)\sqrt{\mu_{m,k}^{1/2}+m^2}\tanh\left(\ell\sqrt{\mu_{m,k}^{1/2}+m^2}\right)\right.\\
\left.-(\mu_{m,k}^{1/2}+(1-\sigma)m^2)\sqrt{\mu_{m,k}^{1/2}-m^2}\tan\left(\ell\sqrt{\mu_{m,k}^{1/2}-m^2}\right) \right)m\cos(mx)\\
=m\left(1+\frac{\mu_{m,k}^{1/2}-(1-\sigma)m^2}{\mu_{m,k}^{1/2}+(1-\sigma)m^2}\right)(\mu_{m,k}^{1/2}-(1-\sigma)m^2)\sqrt{\mu_{m,k}^{1/2}+m^2}\tanh\left(\ell\sqrt{\mu_{m,k}^{1/2}+m^2}\right)\cos(mx)\\
=2m\mu_{m,k}^{1/2}\frac{\mu_{m,k}^{1/2}-(1-\sigma)m^2}{\mu_{m,k}^{1/2}+(1-\sigma)m^2}\sqrt{\mu_{m,k}^{1/2}+m^2}\tanh\left(\ell\sqrt{\mu_{m,k}^{1/2}+m^2}\right)\cos(mx),
\end{multline*}
where in the second equality we used the fact that $\mu_{m,k}$ is the $(k\!-\!1)$th positive zero of the function $\lambda\mapsto\Upsilon^m(\lambda,\ell)$,
see Proposition \ref{defunzioni}. By collecting terms, the proof is concluded observing that
$$\begin{array}{cc}
(1-\sigma)|D^2v_{m,k}|^2+\sigma(\Delta v_{m,k})^2-\mu_{m,k}v^2_{m,k}|_{y=\ell}\\
=(v_{m,k})^2_{xx}|_{y=\ell}+(v_{m,k})^2_{yy}|_{y=\ell}+2\sigma(v_{m,k})_{xx}(v_{m,k})_{yy}|_{y=\ell}-\mu_{m,k}v^2_{m,k}|_{y=\ell}+2(1-\sigma)(v_{m,k})_{xy}^2|_{y=\ell}.
\end{array}$$
The case $k=1$ is analogous.
\end{proof}

\begin{lemma}\label{torsionali}
Let $w_{n,j}$ be the eigenfunction associated with the eigenvalue $\nu_{n,j}$ as in Theorem \ref{eigenvalue}. Then
$$
(1-\sigma)|D^2w_{n,j}|^2+\sigma(\Delta w_{n,j})^2-\nu_{n,j}w^2_{n,j}|_{y=\ell}=\tilde A_{n,j}\sin^2(nx)+\tilde B_{n,j}\cos^2(nx),
$$
where
$$
\tilde A_{n,j}=4\nu_{n,j}(n^4-\sigma^2n^4-\nu_{n,j})
$$
and
$$
\tilde B_{n,j}=4(1-\sigma)n^2\nu_{n,j}(\nu_{n,j}^{1/2}+n^2)\left(\frac{\nu_{n,j}^{1/2}-(1-\sigma)n^2}{\nu_{n,j}^{1/2}+(1-\sigma)n^2}\right)^2
\coth^2\left(\ell\sqrt{n^2+\nu_{n,j}^{1/2}}\right).
$$
\end{lemma}
\begin{proof}
The proof is identical to that  of Lemma \ref{calcoli}.
\end{proof}

Note that from \eq{lowerbound} we deduce that $A_{m,k},\tilde A_{n,j}<0$ and $B_{m,k},\tilde B_{n,j}>0$ for all $m,k,n,j$.

Now we compute the $H^2_*(\Omega)$-norms, see \eq{neumannform1}, of the eigenfunctions characterized in Theorem \ref{eigenvalue}.

\begin{lemma}
Let $\mu_{m,k},\nu_{n,j}$ be the eigenvalues of problem (\ref{bridge}), and let $v_{m,k},w_{n,j}$ be the respective eigenfunctions as described in
Theorem \ref{eigenvalue}. Let also
$$\alpha_{m,k}=\left|m^2-\sqrt{\mu_{m,k}}\right|,  \beta_{m,k}=m^2+\sqrt{\mu_{m,k}}\quad\mbox{and}\quad
\tilde\alpha_{n,j}=\left|n^2-\sqrt{\nu_{n,j}}\right| , \tilde\beta_{n,j}=n^2+\sqrt{\nu_{n,j}}.$$
Then
\begin{eqnarray*}
||v_{m,1}||^2 &=& \frac{\ell\pi\mu_{m,1}}2\frac{(\sigma m^2-\alpha_{m,1})^2}{\cosh^2(\ell\sqrt{\beta_{m,1}})}+
\frac{\ell\pi\mu_{m,1}}2\frac{(\beta_{m,1}-\sigma m^2)^2}{\cosh^2(\ell\sqrt{\alpha_{m,1}})}\\
 & & +\pi\mu_{m,1}(\sigma m^2-\alpha_{m,1})^2\sqrt{\beta_{m,1}}\tanh(\ell\sqrt{\beta_{m,1}})\left(\frac{m^2}{m^4-\mu_{m,1}}+\frac{4(1-\sigma) m^2}{\mu_{m,1}-(1-\sigma)^2m^4}\right),
\end{eqnarray*}
and, for $k>1$,
\begin{eqnarray*}
||v_{m,k}||^2 &=& \frac{\ell\pi\mu_{m,k}}2 \frac{(\sigma m^2+\alpha_{m,k})^2}{\cosh^2(\ell\sqrt{\beta_{m,k}})}
+\frac{\ell\pi\mu_{m,k}}2 \frac{(\beta_{m,k}-\sigma m^2)^2}{\cos^2(\ell\sqrt{\alpha_{m,k}})}\\
&& +\pi\mu_{m,k}(\sigma m^2+\alpha_{m,k})^2\sqrt{\beta_{m,k}}\tanh(\ell\sqrt{\beta_{m,k}})\left(\frac{m^2}{m^4-\mu_{m,k}}+\frac{4(1-\sigma) m^2}{\mu_{m,k}-(1-\sigma)^2m^4}\right).
\end{eqnarray*}
Similarly,
\begin{eqnarray*}
||w_{n,1}||^2 &=& -\frac{\ell\pi\nu_{n,1}}2
\frac{(\sigma n^2-\tilde\alpha_{n,1})^2}{\sinh^2(\ell\sqrt{\tilde\beta_{n,1}})}-\frac{\ell\pi\nu_{n,1}}2\frac{(\tilde\beta_{n,1}-\sigma n^2)^2}{\sinh^2(\ell\sqrt{\tilde\alpha_{n,1}})}\\
& & +\pi\nu_{n,1}(\sigma n^2-\tilde\alpha_{n,1})^2\sqrt{\tilde\beta_{n,1}}
\coth(\ell\sqrt{\tilde\beta_{n,1}})\left(\frac{n^2}{n^4-\nu_{n,1}}+\frac{4(1-\sigma) n^2}{\nu_{n,1}-(1-\sigma)^2n^4}\right),
\end{eqnarray*}
and, for $j>1$,
\begin{eqnarray*}
||w_{n,j}||^2 &=& -\frac{\ell\pi\nu_{n,j}}2 \frac{(\sigma n^2+\tilde\alpha_{n,j})^2}{\sinh^2(\ell\sqrt{\tilde\beta_{n,j}})}
+\frac{\ell\pi\nu_{n,j}}2 \frac{(\tilde\beta_{n,j}-\sigma n^2)^2}{\sin^2(\ell\sqrt{\tilde\alpha_{n,j}})}\\
&&+\pi\nu_{n,j}(\sigma n^2+\tilde\alpha_{n,j})^2\sqrt{\tilde\beta_{n,j}}
\coth(\ell\sqrt{\tilde\beta_{n,j}})\left(\frac{n^2}{n^4-\nu_{n,j}}+\frac{4(1-\sigma) n^2}{\nu_{n,j}-(1-\sigma)^2n^4}\right).
\end{eqnarray*}
\end{lemma}
\begin{proof} Here we provide the computations only for $v_{m,1}$, the other cases being similar.
We recall some identities which will be extensively used in the following computations:
$$
\begin{array}{cc}
\int_0^\pi\sin^2(mx)dx=\int_0^\pi\cos^2(mx)dx=\frac\pi 2\qquad\forall m\in\mathbb N ,\\
\int_{-\ell}^\ell\cosh^2(ay)dy = \frac{\sinh(a\ell)\cosh(a\ell)}{a}+\ell,\\
\int_{-\ell}^\ell\cosh(ay)\cosh(by)dy = \frac{2}{a^2-b^2} (a \sinh(a\ell)\cosh(b\ell)- b\sinh(b\ell)\cosh(a\ell)),
\end{array}
$$
for any $a,b\in\mathbb R$ such that $a^2\neq b^2$. We shall also make use of the implicit characterization $(i)$ following Theorem \ref{eigenvalue}
for the eigenvalue $\mu_{m,1}$ associated with the eigenfunction $v_{m,1}$, that is,
$$
(\beta_{m,1}-\sigma m^2)^2\sqrt{\alpha_{m,1}}\tanh(l\sqrt{\alpha_{m,1}})=(\sigma m^2-\alpha_{m,1})^2\sqrt{\beta_{m,1}}\tanh(l\sqrt{\beta_{m,1}}),
$$
where we have set $\alpha_{m,1}=m^2-\sqrt{\mu_{m,1}} $, $\beta_{m,1}=m^2+\sqrt{\mu_{m,1}}$. Hence
$$\begin{array}{rcl}
\frac2\pi \int_{\Omega}v_{m,1}^2 &=& \frac{(\sigma m^2-\alpha_{m,1})^2}{\cosh^2(\ell\sqrt{\beta_{m,1}})}\int_{-\ell}^\ell\cosh^2(y\sqrt{\beta_{m,1}})dy
+\frac{(\beta_{m,1}-\sigma m^2)^2}{\cosh^2(\ell\sqrt{\alpha_{m,1}})}\int_{-\ell}^\ell\cosh^2(y\sqrt{\alpha_{m,1}})dy\\
&& +2\frac{(\sigma m^2-\alpha_{m,1})(\beta_{m,1}-\sigma m^2)}{\cosh(\ell\sqrt{\alpha_{m,1}})\cosh(\ell\sqrt{\beta_{m,1}})}\int_{-\ell}^\ell\cosh(y\sqrt{\alpha_{m,1}})\cosh(y\sqrt{\beta_{m,1}})dy\\
&=& \frac{(\sigma m^2-\alpha_{m,1})^2}{\sqrt{\beta_{m,1}}}\tanh(\ell\sqrt{\beta_{m,1}})+\frac{\ell(\sigma m^2-\alpha_{m,1})^2}{\cosh^2(\ell\sqrt{\beta_{m,1}})}
+\frac{(\beta_{m,1}-\sigma m^2)^2}{\sqrt{\alpha_{m,1}}}\tanh(\ell\sqrt{\alpha_{m,1}})+\frac{\ell(\beta_{m,1}-\sigma m^2)^2}{\cosh^2(\ell\sqrt{\alpha_{m,1}})}\\
&&+4\frac{(\sigma m^2-\alpha_{m,1})(\beta_{m,1}-\sigma m^2)}{\alpha_{m,1}-\beta_{m,1}}\left(\sqrt{\alpha_{m,1}}\tanh(\ell\sqrt{\alpha_{m,1}})-\sqrt{\beta_{m,1}}\tanh(\ell\sqrt{\beta_{m,1}})\right)\\
&=& \frac{\ell(\sigma m^2-\alpha_{m,1})^2}{\cosh^2(\ell\sqrt{\beta_{m,1}})}+\frac{\ell(\beta_{m,1}-\sigma m^2)^2}{\cosh^2(\ell\sqrt{\alpha_{m,1}})}
+(\sigma m^2-\alpha_{m,1})^2\sqrt{\beta_{m,1}}\tanh(\ell\sqrt{\beta_{m,1}})\left(\frac{1}{\beta_{m,1}}+\frac{1}{\alpha_{m,1}}\right)\\
&&+(\sigma m^2-\alpha_{m,1})^2\sqrt{\beta_{m,1}}\tanh(\ell\sqrt{\beta_{m,1}})\left(\frac{\sigma m^2-\alpha_{m,1}}{\beta_{m,1}-\sigma m^2}-\frac{\beta_{m,1}-\sigma m^2}{\sigma m^2-\alpha_{m,1}}\right)\frac{4}{\alpha_{m,1}-\beta_{m,1}}\\
&=&\frac{\ell(\sigma m^2-\alpha_{m,1})^2}{\cosh^2(\ell\sqrt{\beta_{m,1}})}+\frac{\ell(\beta_{m,1}-\sigma m^2)^2}{\cosh^2(\ell\sqrt{\alpha_{m,1}})}
+(\sigma m^2-\alpha_{m,1})^2\sqrt{\beta_{m,1}}\tanh(\ell\sqrt{\beta_{m,1}})\frac{2m^2}{m^4-\mu_{m,1}}\\
&&+(\sigma m^2-\alpha_{m,1})^2\sqrt{\beta_{m,1}}\tanh(\ell\sqrt{\beta_{m,1}})\frac{8(1-\sigma) m^2}{\mu_{m,1}-(1-\sigma)^2m^4}.
\end{array}$$
Since $||v||^2=\lambda||v||_{L^2(\Omega)}^2$ for any eigenfunction $v\in H^2_*(\Omega)$ associated with the eigenvalue $\lambda$, this concludes the proof.\end{proof}

\section{Conclusions}\label{conc}

Following suggestions from the engineering literature \cite{jhnm}, we have set up an analytic approach to the study of the behavior of frequencies in partially
hinged rectangular plates subject to shape variations. This is usually considered a very difficult task, see \cite[Chapter 6]{jhnm}, and several
simplifications in the model are necessary, provided they maintain the behavior of the structure. In particular, our approach does
not take into account the elastic deformation of the plate which is however considered a negligible phenomenon, see again \cite{jhnm}. The present paper
should be considered as a first simple attempt to analyze the torsional stability of a deck with respect to shape variations. We do not pretend it to be
exhaustive, it may certainly be improved. But, at least, it gives very strong hints on the impossibility of having a simple and reliable formula able to
quantify the torsional stability, and raises severe criticisms against the formulas available in literature.\par
From Theorem \ref{eigenvalue} we learn that a longitudinal oscillation may be a linear combination of infinitely many different eigenfunctions.
For instance, a linear combination of eigenfunctions associated with $\mu_{m,k}$ for $k\ge1$ has the form $Y(y)\sin(mx)$ with $Y$ being an even function
of $y$. This means that any such function is approximately of the form $C\sin(mx)$ and, in turn, a longitudinal oscillation of the plate. A similar argument works also for torsional
oscillations. Therefore, a precise characterization of longitudinal or torsional normal modes is not straightforward. However, the extremely larger frequencies
that appear in Table \ref{tableigen} suggest to restrict the attention to the eigenvalues in \eq{comparison}. In this respect, let us recall what happened
at the Tacoma Bridge. A few days prior to its collapse, the project engineer L.R.\ Durkee wrote a letter (see \cite[p.28]{ammann}) describing the
oscillations which were observed so far. He wrote: {\em Altogether, seven different motions have been definitely identified on the main
span of the bridge, and likewise duplicated on the model. These different wave actions consist of motions from the simplest, that of no nodes, to the
most complex, that of seven modes}. Moreover, Farquharson \cite[V-10]{ammann} witnessed the collapse and wrote that {\em the motions, which a
moment before had involved a number of waves (nine or ten) had shifted almost instantly to two}. This means that the instability occurred from the ninth
or tenth longitudinal mode to the second torsional one. Smith and Vincent \cite[p.21]{tac2} state that this shape of torsional oscillations is the
only possible one, see also \cite[Section 1.6]{bookgaz} for further evidence and more historical facts. For these reasons, the eigenvalues in \eq{comparison}
appear more than enough for a reliable stability analysis.\par
On the other hand, from \cite[Section 2.5]{banerjee} we learn that the flutter speed $V_c$ of a bridge depends on the couple of modes considered.
All this suggests to study the torsional stability of the plate for several (but not all) couples of (longitudinal,torsional) normal modes.
When the parameters are fixed according to the Tacoma Narrows Bridge data, see \eq{parameters}, the most interesting ones are $(\bar\mu_{m,1},\bar\nu_{2,2})$ for $m=1,...,14$ since they correspond to low frequencies, they satisfy $\mu<\nu$,
and they involve the second torsional mode.\par
In this paper we have generalized the definition of critical energy (or flutter velocity), see formula \eq{flutterocardgen}, in order to make it also
usable for plates having shapes other than rectangular, see \eq{newkind}. Then we discovered several rules that the critical energy of a couple
of (longitudinal,torsional) eigenvalues is supposed to satisfy: these are summarized in \eq{suggest1}, \eq{suggest3}, \eqref{reasonable}, and \eq{suggest4}.
Nevertheless, the empirical formula \eq{odeigen} shows that these suggestions cannot be simultaneously satisfied.
From Section \ref{flutter} we learn that a fully accepted way to compute the flutter velocity is not available and that the most popular formula coming
from engineering literature is not reliable. Clearly, a formula able to quantify the torsional stability of a deck would be very important for the safety of bridges and any
progress in this direction would be extremely welcome.
But our conclusion is that it cannot be found in an explicit and simple form.\par
Finally, let us mention a related challenging problem. The deck of a bridge is often strengthened with stiffening trusses \cite{haer,xant}.
It would be of great interest to study the variation of the frequencies of oscillations in presence of trusses. In this respect, the framework could
be that of Michell trusses \cite{michell}, see also \cite{truss} for a first naif attempt. Related recent works are \cite{BouButSep,Bouche,bougang},
which deal with second order energies (leading to fourth order equations such as \eq{bridge}) on thin structures such as a plate modeling the deck of
a bridge.\par\bigskip\noindent	
\textbf{Acknowledgments.} The authors are grateful to Luigi Provenzano for several discussions.
The first and second author are partially supported by the Research Project FIR (Futuro in Ricerca) 2013 \emph{Geometrical and
qualitative aspects of PDE's}. The third author is partially supported by the PRIN project {\em Equazioni alle derivate parziali di tipo ellittico e parabolico:
aspetti geometrici, disuguaglianze collegate, e applicazioni}. The three authors are members of the Gruppo Nazionale per l'Analisi Matematica, la Probabilit\`a
e le loro Applicazioni (GNAMPA) of the Istituto Nazionale di Alta Matematica (INdAM).

\end{document}